\documentclass[12pt]{amsart}
\usepackage[margin=2.5cm]{geometry}
\usepackage{amsmath,amssymb,amsfonts,amsthm,amsrefs,mathrsfs}
\usepackage{epic}
\usepackage{amsopn,amscd,graphicx}
\usepackage{color,transparent} 
\usepackage[usenames, dvipsnames]{xcolor}
\usepackage
[colorlinks, breaklinks,
bookmarks = false,
linkcolor = NavyBlue,
urlcolor = ForestGreen,
citecolor = ForestGreen,
hyperfootnotes = false
]
{hyperref}
\usepackage{scalerel}
\usepackage{stackengine,wasysym}
\usepackage{palatino, mathpazo}
\usepackage{multirow}
\usepackage[all,cmtip]{xy}
\usepackage{stmaryrd}
\usepackage{tikz-cd}
\usepackage{yhmath}
\usepackage{enumerate}
\usepackage[dvipsnames]{xcolor}
 1
 1
 1

\newtheorem{theorem}{Theorem}[section]

\newtheorem{lemma}[theorem]{Lemma}
\newtheorem{remark}[theorem]{Remark}

\newtheorem{example}[theorem]{Example}
\theoremstyle{definition}
\newtheorem{definition}[theorem]{Definition}

\newtheoremstyle{named}{}{}{\itshape}{}{\bfseries}{.}{.5em}{\thmnote{#3}#1}
\theoremstyle{named}

\numberwithin{equation}{subsection}

\newcommand{\ol}{\overline}
\newcommand{\R}{\mathbb{R}}
\newcommand{\C}{\mathbb{C}}
\newcommand{\N}{\mathbb{N}}

\newcommand{\ddbar}{\overline\partial}
\newcommand{\pr}{\partial}
\newcommand{\Td}{\widetilde}

\newcommand{\abs}[1]{\left\vert#1\right\vert}
\newcommand{\set}[1]{\left\{#1\right\}}
\newcommand{\To}{\rightarrow}
\def\cC{\mathscr{C}}
\newcommand\reallywidetilde[1]{\ThisStyle{%
  \setbox0=\hbox{$\SavedStyle#1$}%
  \stackengine{-.1\LMpt}{$\SavedStyle#1$}{%
    \stretchto{\scaleto{\SavedStyle\mkern.2mu\AC}{.5150\wd0}}{.6\ht0}%
  }{O}{c}{F}{T}{S}%
}}
\title[On the Second Coefficient in the Semi-Classical Expansion of Toeplitz Operators]{On the Second Coefficient in the Semi-Classical Expansion of Toeplitz Operators}

\begin{document}

\author[Chin-Chia Chang]{Chin-Chia Chang}
\address{Universit{\"a}t zu K{\"o}ln,  Mathematisches Institut,
	Weyertal 86-90, 50931 K{\"o}ln, Germany}
\thanks{Chin-Chia Chang was partially supported by the DFG project  
	SFB/TRR 191 (Project-ID 281071066-TRR~191)}
\email{cchang@math.uni-koeln.de}

\author[Hendrik Herrmann]{Hendrik Herrmann}
\address{University of Vienna, Faculty of Mathematics, Oskar-Morgenstern-Platz 1, 1090 Vienna, Austria}
\thanks{Hendrik Herrmann was partially supported the by the Austrian Science Fund (FWF) grant: The dbar Neumann operator and related topics (DOI 10.55776/P36884)}
\email{hendrik.herrmann@univie.ac.at or post@hendrik-herrmann.de}

\author[Chin-Yu Hsiao]{Chin-Yu Hsiao}
\address{Department of Mathematics, National Taiwan University, Astronomy-Mathematics Building, 
No. 1, Sec. 4, Roosevelt Road, Taipei 10617, Taiwan}
\thanks{Chin-Yu Hsiao was partially supported by the NSTC Project 113-2115-M-002 -011 -MY3\\
This preprint has not undergone peer review or any post-submission improvements or corrections. The Version of Record of this article is published in \textit{Analysis and Mathematical Physics}, and is available online at https://doi.org/10.1007/s13324-025-01105-2}
\email{chinyuhsiao@ntu.edu.tw or chinyu.hsiao@gmail.com} 

\date{\today}

\begin{abstract}  Let $X$ be a compact strictly pseudoconvex embeddable 
CR manifold and let
$A$ be the Toeplitz operator on $X$ associated with a Reeb vector field $\mathcal{T}\in\cC^\infty(X,TX)$. Consider the operator $\chi_k(A)$ 
defined by the functional
calculus of $A$, where $\chi$ is a smooth function
with compact support in the positive real line and 
$\chi_k(\lambda):=\chi(k^{-1}\lambda)$. It was established recently that $\chi_k(A)(x,y)$ admits a full asymptotic expansion in $k$ when \(k\) becomes large. The second coefficient of the expansion plays an important role in the further studies of CR geometry. In this work, we calculate the second coefficient of the expansion.
\end{abstract}
\maketitle 
\tableofcontents


\tableofcontents
\section{Introduction and Main Results} \label{s-gue241206yyd}
{\tiny{{\color{white}{\subsection{ }}}}}
Let $X$ be a compact strictly pseudoconvex embeddable 
CR manifold and let
$A$ be the Toeplitz operator on $X$ associated with a Reeb vector field $\mathcal{T}\in\cC^\infty(X,TX)$. Consider the operator $\chi_k(A)$ 
defined by the functional
calculus of $A$, where $\chi$ is a smooth function
with compact support in the positive real line and 
$\chi_k(\lambda):=\chi(k^{-1}\lambda)$. Let $\chi_k(A)(x,y)$ denote the distribution kernel of $\chi_k(A)$. It was shown in~\cite{HHMS23}*{Theorem 1.1} that $\chi_k(A)(x,y)$ admits a full asymptotic expansion in $k$. In particular, 
\begin{equation}\label{e-gue241114yyd}
\chi_k(A)(x,x)\sim k^{n+1}b_0(x)+k^nb_1(x)+\cdots,
\end{equation}
where $b_j(x)\in\mathcal{C}^\infty(X)$, $j=0,1,\ldots$. The expansion \eqref{e-gue241114yyd} can be seen as a CR analogue of the Tian-Yau-Catlin-Zelditch Bergman kernel asymptotics in complex geometry (we refer the reader to the book~\cite{Ma_Marinescu_HMI_BKE_2007}). The term $b_0$ in \eqref{e-gue241114yyd} was already calculated in~\cite{HHMS23}. For the Sasakian case the expansion~\eqref{e-gue241114yyd} was obtained in~\cite{Herrmann_Hsiao_Li_T_equivariant_Szego_2020}. It is a very natural (and fundamental question) to calculate $b_1$ in \eqref{e-gue241114yyd}. The calculation of $b_1$ can be seen as some kind of local index theorem in CR geometry. On the other hand, the second coefficient of Bergman kernel asymptotic expansion in complex geometry plays an important role in K\"ahler geometry (see \cite{Ma_Marinescu_HMI_BKE_2007}). 
The main goal of this work is to calculate $b_1$. 
 
We now formulate our main result. We refer the reader to Section~\ref{s-gue241119yyd} for some terminology and notations used here. 
Let $(X,T^{1,0}X)$ be a compact orientable strictly pseudoconvex embeddable CR manifold of dimension $2n+1$, \(n\geq 1\). Fix a global one form $\xi\in\cC^\infty(X,T^*X)$ such that for any \(x\in X\) we have $\xi(x)\neq0$,  $\ker\xi(x)=\operatorname{Re}T_x^{1,0}X$ and the Levi form $\mathcal{L}_x$ is positive definite on \(T^{1,0}_xX\). Here the Levi form $\mathcal{L}_x=\mathcal{L}^\xi_x$ at \(x\in X\) (with respect to \(\xi\)) is defined by 
\[\mathcal{L}_x\colon T^{1,0}_xX\times T^{1,0}_xX\to\C,\,\,\,\mathcal{L}_x(U,V)=\frac{1}{2i}d\xi(U,\overline{V}).\]
 There exists a unique vector field $\mathcal{T}\in\cC^\infty(X,TX)$ such that $\mathcal{T}\lrcorner\xi\equiv1$ and $\mathcal{T}\lrcorner d\xi\equiv 0$, where \(\lrcorner\) denotes the interior product (contraction) between vectors and forms. We call $\mathcal{T}$  the Reeb vector field on $X$ with respect to $\xi$. Let $\langle\,\cdot\,|\,\cdot\,\rangle$ denote the Hermitian metric on $\mathbb CTX$ 
 induced by the Levi form $\mathcal{L}_x$ (see \eqref{e-gue241115yyd}). The Hermitian $\langle\,\cdot\,|\,\cdot\,\rangle$ on $\mathbb CTX$ induces a Hermitian metric 
$\langle\,\cdot\,|\,\cdot\,\rangle$ on $\oplus_{p,q\in\mathbb N_0,0\leq p,q\leq n}T^{*p,q}X$, where $T^{*p,q}X$ denotes the bundle of $(p,q)$ forms of $X$. Let $\Omega^{p,q}(X)$
denote the space of smooth sections with values in $T^{*p,q}X$. We write $\cC^\infty(X):=\Omega^{0,0}(X)$.

 Fix a volume form $dV$ on $X$ such that $\mathcal{L}_{\mathcal{T}}dV\equiv 0$, where $\mathcal{L}_{\mathcal{T}}$ denotes the Lie derivative of $\mathcal{T}$. 
 For $p, q\in\mathbb N_0$, $0\leq p,q\leq n$, let $(\,\cdot\,|\,\cdot\,)$ be the $L^2$ inner product on $\Omega^{p,q}(X)$ induced by $dV$ and $\langle\,\cdot\,|\,\cdot\,\rangle$ and let $L^2_{(p,q)}(X)$ be the completion of $\Omega^{p,q}(X)$ with respect to $(\,\cdot\,|\,\cdot\,)$. We write $L^2(X):=L^2_{(0,0)}(X)$. Let $\ddbar_b: \cC^\infty(X)\To\Omega^{0,1}(X)$ 
 be the tangential Cauchy-Riemann operator. We extend $\ddbar_b$ to the $L^2$ space:
 \[\ddbar_b: {\rm Dom\,}\ddbar_b\subset L^2(X)\To L^2_{(0,1)}(X),\]
 where ${\rm Dom\,}\ddbar_b:=\set{u\in L^2(X);\, \ddbar_bu\in L^2_{(0,1)}(X)}$. Let $H^0_b(X):={\rm Ker\,}\ddbar_b$. Let \[\Pi: L^2(X)\to H^0_b(X)\]
 be the Szeg\H{o} projection, i.e., the orthogonal projection onto $H^0_b(X)$ with respect to $(\,\cdot\,|\,\cdot\,)$. Let $\Pi(x,y)\in\mathcal{D}'(X\times X)$ denote the distribution kernel of $\Pi$. We recall
Boutet de Monvel--Sj\"ostrand's fundamental theorem~\cite{Boutet_Sjoestrand_Szego_Bergman_kernel_1975} (see also~\cite{Hsiao_Szego_Bergman_kernel_q_forms_2010},~\cite{Hsiao_Marinescu_Szego_lower_energy_2017}) about the structure of the 
Szeg\H{o} kernel.
For any coordinate patch $(D,x)$ on $X$ and any positive function \(\lambda\in\mathscr{C}^\infty(D,\R_+)\)
there is a smooth function $\varphi:D\times D\to\C$ with
\begin{equation}
\label{Eq:PhaseFuncMainThm}
\begin{split}
&\operatorname{Im}\varphi(x,y)\geq 0,\\
&\varphi(x,y)=0~\text{if and only if}~y=x,\\
&d_x\varphi(x,x)=-d_y\varphi(x,x)=\lambda(x)\xi(x),\ \,
\end{split}
\end{equation} 
such that we have on $D\times D$, the Szeg\H{o} projector
can be approximated by a Fourier integral operator
\begin{equation}\label{eq:PiFIO}
\Pi(x,y)=\int_0^{+\infty} 
e^{it\varphi(x,y)}s(x,y,t)dt+F(x,y),
\end{equation}
where $F(x,y)\in\cC^\infty(D\times D)$ and $s(x,y,t)
\in S^{n}_{\operatorname{cl}}(D\times D\times{\R}_+)$ 
is a classical H\"ormander symbol satisfying $s(x,y,t)
\sim\sum_{j=0}^{+\infty}s_j(x,y)t^{n-j}$ in 
$S^{n}_{1,0}(D\times D\times{\R}_+)$, $s_j(x,y)\in\cC^\infty(D\times D)$, $j=0,1,\ldots$, and if $\lambda(x)=1$,
\begin{equation}
\label{eq:leading term s_0 intro}
s_0(x,x)=\frac{1}{2\pi^{n+1}}\frac{dV_\xi}{dV}(x),
\end{equation} 
 where 
 \begin{equation}\label{e-gue241115ycd}
 dV_{\xi}:=\frac{2^{-n}}{n!}\xi\wedge \left(d\xi\right)^n.
 \end{equation} 
 
Let 
 \[A:=\Pi(-i\mathcal{T})\Pi: \cC^\infty(X)\To\cC^\infty(X)\]
 be the Toeplitz operator. We extend $A$ to the $L^2$ space:
 \begin{eqnarray}\label{eq:DefExtensionToeplitzOperator}
 	A:=\Pi(-i\mathcal{T})\Pi: {\rm Dom\,}A\subset L^2(X)\To L^2(X),
 \end{eqnarray}
  where ${\rm Dom\,}A:=\set{u\in L^2(X);\, Au\in L^2(X)}$. It was shown in~\cite{HHMS23}*{Theorem 3.3} that $A$ is a self-adjoint operator. We  consider the operator $\chi_k(A)$ 
defined by functional
calculus of $A$, where $\chi$ is a smooth function
with compact support in the positive real line and 
$\chi_k(\lambda):=\chi(k^{-1}\lambda)$.  Let $\chi_k(A)(x,y)$ denote the distribution kernel of $\chi_k(A)$. We recall the following consequence from the result obtained in~\cite{HHMS23}. 
 
 \begin{theorem}\label{thm:ExpansionMain}
Let $(X,T^{1,0}X)$ be a compact orientable strictly pseudoconvex embeddable CR manifold of dimension $2n+1$, \(n\geq 1\). Fix a global one form $\xi\in\cC^\infty(X,T^*X)$ such that for any \(x\in X\) we have $\xi(x)\neq0$,  $\ker\xi(x)=\operatorname{Re}T_x^{1,0}X$ and the respective Levi form $\mathcal{L}_x$ is positive definite. Denote by \(\mathcal{T}\) the Reeb vector field associated to \(\xi\), choose a volume form \(dV\) on \(X\) with $\mathcal{L}_{\mathcal{T}}dV\equiv 0$  and consider the operator \(A\) given by  \eqref{eq:DefExtensionToeplitzOperator}.  
Let $(D,x)$ be any coordinate patch and let 
$\varphi:D\times D\to\C$ be any phase function satisfying 
\eqref{Eq:PhaseFuncMainThm} and \eqref{eq:PiFIO}.
Then,  for any \(\chi\in \mathscr{C}^\infty_c(\R_+)\) putting \(\chi_k(t):=\chi(k^{-1}t)\), \(k,t\in\R_+\), one has for the distributional kernel $\chi_k(A)(x,y)$ of $\chi_k(A)$ that
\begin{equation}
\label{eq:asymptotic expansion of chi_k(T_P)}
\chi_k(A)(x,y)=\int_0^{+\infty} 
e^{ikt\varphi(x,y)}b^\chi(x,y,t,k)dt+O\left(k^{-\infty}\right)~\text{on}~D\times D,
\end{equation}
where $b^\chi(x,y,t,k)\in S^{n+1}_{\mathrm{loc}}
(1;D\times D\times{\R}_+)$,
\begin{equation}
\label{Eq:LeadingTermMainThm}
\begin{split}
&b^\chi(x,y,t,k)\sim\sum_{j=0}^{+\infty}b^\chi_{j}(x,y,t)k^{n+1-j}~
\text{in $S^{n+1}_{\mathrm{loc}}(1;D\times D\times{\R}_+)$},\\
&b^\chi_j(x,y,t)\in\mathscr{C}^\infty(D\times D\times{\R}_+),~j=0,1,2,\ldots,\\
&b^\chi_{0}(x,x,t)=\frac{1}{2\pi ^{n+1}}
\frac{dV_{\xi}}{dV}(x)\,\chi(t)\,t^n,\ \ \mbox{for every $x\in D$ with \(\lambda(x)=1\) (see \eqref{Eq:PhaseFuncMainThm})},
\end{split}
\end{equation}
with \(dV_\xi\) given by \eqref{e-gue241115ycd} and for $I^\chi:=\operatorname{supp}\chi$,
\begin{equation}\label{e-gue241115ycda}
\begin{split}
{\rm supp\,}_t b^\chi(x,y,t,k),~{\rm supp\,}_t b^\chi_j(x,y,t)\subset I^\chi,\ \ j=0,1,2,\ldots~.
\end{split}
\end{equation}
(see Definition~\ref{def:suppt} for the meaning of \(\operatorname{supp}_t\).)

In particular, we have 
\begin{equation}\label{e-gue241115ycdI}
\chi_k(A)(x,x)\sim\sum^{+\infty}_{j=0}b^\chi_j(x)k^{n+1-j}\ \ \mbox{in $S^{n+1}_{{\rm loc\,}}(1;X)$}, 
\end{equation}
where $b^\chi_j(x)=\int b^\chi_j(x,x,t)dt$, $b^\chi_j(x,y,t)$ is as in \eqref{Eq:LeadingTermMainThm}, $j=0,1,\ldots$. 
\end{theorem} 

The expansion \eqref{eq:asymptotic expansion of chi_k(T_P)} can be seen as a semi-classical version of Boutet de Monvel-Sj\"ostrand's result on strictly pseudoconvex CR manifolds (see \eqref{eq:PiFIO}). As mentioned above, it is a very natural (and fundamental question) to calculate $b^\chi_1(x)$ in \eqref{e-gue241115ycdI}. In this work, we successfully calculate $b^\chi_1(x)$. The following is our main result. 

\begin{theorem}\label{t-gue241115yyd}
With the notations and assumptions in Theorem~\ref{thm:ExpansionMain}, for $b^\chi_1(x)$ in \eqref{e-gue241115ycdI}, we have 
\begin{equation}\label{e-gue241115ycdb}
b^\chi_1(x)=\frac{1}{2\pi^{n+1}}e^{-2(n+1)f(x)}\left(\frac{1}{2}R_{scal}(x)+(n+1)\Delta_b f(x)\right)\int\chi(t)t^{n-1}dt,
\end{equation}
where $f(x):=\frac{1}{2(n+1)}\log \frac{dV(x)}{dV_\xi(x)}\in\cC^\infty(X)$ with \(dV_\xi\) given by \eqref{e-gue241115ycd}, $R_{scal}$ is the Tanaka-Webster scalar curvature with respect to the contact form $\xi$ (see \eqref{e-gue241122yydI}) and $\Delta_b:\cC^{\infty}(X)\rightarrow\cC^{\infty}(X)$ denotes the CR sublaplacian operator (see \eqref{e-gue241122yyd}). 
\end{theorem} 

To prove Theorem~\ref{t-gue241115yyd}, we need to understand how the symbol $b^\chi_j(x,y,t)$ in \eqref{Eq:LeadingTermMainThm} depends on \(\chi\) and $t$. We have the following.

\begin{theorem}\label{t-gue241115ycd}
With the notations and assumptions in Theorem~\ref{thm:ExpansionMain}, let $(D,x)$ be any coordinate patch and let 
$\varphi:D\times D\to\C$ be any phase function satisfying 
\eqref{Eq:PhaseFuncMainThm} and \eqref{eq:PiFIO}. Then there exist smooth functions
\begin{eqnarray}\label{Eq:Defa_js}
a_{j,s}(x,y)\in\cC^\infty(D\times D),\ \ j,s=0,1,\ldots,
\end{eqnarray}

such that for any \(\chi\in \mathscr{C}^\infty_c(\R_+)\) we can take $b^\chi_j(x,y,t)$ in \eqref{Eq:LeadingTermMainThm}, $j=0,1,\ldots$, so that 
\begin{equation}\label{eqn:ThmI1z}
    b^\chi_j(x,y,t)=\sum^{+\infty}_{s=0}a_{j,s}(x,y)\chi^{(s)}(t)t^{n+s-j},\quad j=0,1,\ldots,
\end{equation}
where \(\chi^{(s)}:=(\partial/\partial t)^s\chi\) and the infinite sum in \eqref{eqn:ThmI1z} converges uniformly in $\cC^{\infty}(K\times K\times I)$ topology, for any compact subset $K\subset D$, $I\subset\mathbb R$, for all $j=0,1\dots$. 

Assume in addition that $\mathcal{T}$ is a CR vector field, that is, $[\mathcal{T},\cC^\infty(X,T^{1,0}X)]\subset\cC^\infty(X,T^{1,0}X)$.  Then, there exists a phase function 
$\varphi:D\times D\to\C$  satisfying 
\eqref{Eq:PhaseFuncMainThm} and \eqref{eq:PiFIO} such that 
$\mathcal{T}_x\varphi(x,y)\equiv1$, $\mathcal{T}_y\varphi(x,y)\equiv-1$, and for any such phase function we can take the $a_{j,s}(x,y)$ in \eqref{Eq:Defa_js}, $j,s=0,1,\ldots$, such that for all \(j\geq 0\) we have
\begin{equation}\label{e-gue241116yyd}
\begin{split}
&a_{j,s}(x,y)=0,\ \ s=1,2,\ldots,\\
&\mathcal{T}_xa_{j,0}(x,y)=\mathcal{T}_ya_{j,0}(x,y)=0. 
\end{split}
\end{equation}
\end{theorem}

\begin{remark}\label{r-gue241119yyd}
(i) In the proof of Theorem~\ref{t-gue241115yyd}, we show that for a fix point $p\in X$, we can find a phase $\varphi$ satisfying \eqref{Eq:PhaseFuncMainThm} and \eqref{eq:PiFIO} and $a_{j,s}(x,y)$, $j, s=0,1,\ldots$, such that $a_{1,s}(p,p)=0$, for all $s=1,\ldots$, and then we calculate $a_{1,0}(p,p)$ (see Theorem~\ref{thm:s1_indep_of_deri_of_chi} for the details), where $a_{j,s}(x,y)$, $j, s=0,1,\ldots$, are as in \eqref{eqn:ThmI1z}. 

(ii) To get \eqref{e-gue241115ycdb}, we need to generalize the result in~\cite{Hsiao_Shen_2nd_coefficient_BS_2020} to a more general setting (see Theorem~\ref{t-gue241122yyd} for the details). 

(iii) The first part of Theorem~\ref{t-gue241115ycd} holds for a more general choice of the Toeplitz operator \(A\) (see Theorem~\ref{t-gue241123yyd} for the details).  
\end{remark}

From Theorem~\ref{t-gue241115yyd}, Theorem~\ref{t-gue241115ycd} and by using integration by parts in $t$, we can reformulate our Theorem~\ref{t-gue241115yyd}.

\begin{theorem}\label{t-gue241119yyd}
With the notations and assumptions in Theorem~\ref{thm:ExpansionMain}, there exist smooth functions \(\{a_j\}_{j\geq 0}\subset \mathscr{C}^\infty(X)\)  
 such that for any \(\chi\in \mathscr{C}_c^\infty(\R_+)\) we have 
 \begin{equation}\label{e-gue241119yyd}
\chi_k(A)(x,x)\sim\sum^{+\infty}_{j=0}k^{n+1-j}a_j(x)\int\chi(t)t^{n-j}dt\ \ \mbox{in $S^{n+1}_{{\rm loc\,}}(1;X)$}.
 \end{equation}
 Furthermore, we have 
 \begin{equation}
 \begin{split}
 &a_0(x)=\frac{1}{2\pi ^{n+1}}
\frac{dV_{\xi}}{dV}(x),\\
&a_1(x)=\frac{1}{2\pi^{n+1}}e^{-2(n+1)f(x)}\left(\frac{1}{2}R_{scal}(x)+(n+1)\Delta_b f(x)\right),
\end{split}
\end{equation}
where $f(x):=\frac{1}{2(n+1)}\log \frac{dV(x)}{dV_\xi(x)}\in\cC^\infty(X)$ with \(dV_\xi\) given by \eqref{e-gue241115ycd}, $R_{scal}$ is the Tanaka-Webster scalar curvature with respect to the contact form $\xi$ (see \eqref{e-gue241122yydI}) and $\Delta_b:\cC^{\infty}(X)\rightarrow\cC^{\infty}(X)$ denotes the CR sublaplacian operator (see \eqref{e-gue241122yyd}). 
\end{theorem}
We note that the formula for the coefficient \(a_0\) is a direct consequence from the results in \cite{HHMS23}. Furthermore, it was also shown there that for the circle bundle case the expansion of the form described in~\eqref{e-gue241119yyd} holds (see Example~\ref{Ex:RelationToBKExpansion}). Hence, Theorem~\ref{t-gue241119yyd} extends the results in \cite{HHMS23} by showing that the expansion of the form described in~\eqref{e-gue241119yyd} holds under more general assumptions and provides an explicit formula for \(a_1\). 
\begin{remark}\label{rmk:ajLocalInvariants}
	The functions \(a_j\), \(j\geq 0\), in Theorem~\ref{t-gue241119yyd} are uniquely determined by~\eqref{e-gue241119yyd}.
		Furthermore, for any \(j=0,1,\ldots\) and any point \(x\in X\) the value \(a_j(x)\) of the function \(a_j\) in Theorem~\ref{t-gue241119yyd} is determined by the restriction of \(T^{1,0}X\), \(\xi\) and \(dV\) to any neighborhood of \(x\). More precisely, we have the following: Given a compact CR manifold \((X',T^{1,0}X')\) with one form $\xi'\in\cC^\infty(X',T^*X')$ and volume form \(dV'\) such that \((X',T^{1,0}X')\),  \(\xi'\) and \(dV'\) satisfy the assumptions in Theorem~\ref{thm:ExpansionMain} denote by \(\{a'_j\}_{j\geq0}\subset \mathscr{C}^\infty(X')\) the smooth functions given by Theorem~\ref{t-gue241119yyd} associated to the data  \((X',T^{1,0}X')\), \(\xi'\) and \(dV'\). Let \(D\subset X\) be an open neighborhood of a point \(x\in X\) and let \(D'\subset X'\) be an open neighborhood of the point \(x'\in X'\). If  \(F\colon D'\to D\) is a CR diffeomorphism   with \(F(x')=x\), \(F^*\xi=\xi'\) and \(F^*dV=dV'\) on \(D'\) we have \(a_j'(x')=a_j(x)\) for \(j=0,1,\ldots\) (see Remark~\ref{rmk:verifyargumentrmklocalinvariant}). In particular, choosing \(dV=dV_\xi\) we have that the functions \(a_j\), \(j=0,1,\ldots\), become pseudo-Hermitian invariants for the triple \((X,T^{1,0}X,\xi)\).  
\end{remark}
\begin{example}\label{Ex:RelationToBKExpansion}
	Let \((L,h)\to M\) be a positive holomorphic line bundle with Hermitian metric \(h\) over a compact complex manifold \(M\) of dimension \(\dim_\C M=n\) and volume form \(dV_M\). Denote the Bergman kernel function for \(H^0(M,L^k)\), \(k\in \N\), with respect to the \(L^2\) norm
	\[\|s\|_k:=\sqrt{\int_{M}|s|^2_{h^k}dV_M},\,\,\,\, s\in H^0(M,L^k)\]
	 by \(P_k\) that is \(P_k(x):=\sup\{|s(x)|^2_{h^k}\mid s\in H^0(M,L^k), \|s\|_k\leq1\}\), \(x\in M\). Here \(L^k\) denotes the \(k\)-th tensor power of \(L\) and \(h^k\) is the Hermitian metric on \(L^k\) induced by \(h\). It is well known (see \cite{Catlin_BKE_1999}, \cite{Ma_Marinescu_HMI_BKE_2007}, \cite{Zelditch_BKE_1998}) that \(P_k\) has an asymptotic expansion
	\[P_k(x)\sim\sum^{+\infty}_{j=0}k^{n-j}b_j(x)\ \ \mbox{in $S^{n}_{{\rm loc\,}}(1;M)$},\]
	for smooth functions \(b_j\in\cC^\infty(M,\R)\), \(j\geq 0\).
	Put
	\[X:=\{v\in L^*\mid |v|_{h^*}=1\} \subset L,\,\,\,T^{1,0}X:=\C TX\cap T^{1,0}L^*.\]
	Here \(L^*\) denotes the dual line bundle with metric \(h^*\) induced by \(h\). Then \((X,T^{1,0}X)\) is a compact strongly pseudoconvex CR manifold with a transversal CR \(S^1\)-action given by \(S^1\times X\ni(\lambda,v)\mapsto \lambda v \in X\). Denote by \(\mathcal{T}\) the infinitesimal generator of this action and let \(\xi\in \cC^\infty(X,T^*X)\) be the uniquely defined real one form with \(\mathcal{T}\lrcorner\xi\equiv1\) and \(\xi(T^{1,0}X)=0\). Let \(\pi\colon X\to M\) denote the projection map. Putting \(dV_X=\xi\wedge\pi^*dV_M\) it follows that the setup satisfies the assumptions in Theorem~\ref{t-gue241119yyd}. Hence there exist smooth functions \(a_j\in\cC^\infty(X,\R)\), \(j\geq 0\), such that for any \(\chi\in \cC_c^\infty(\R_+)\) we have
	\begin{eqnarray}\label{eq:ExpansionOnCircleBundleCase}
		\chi_k(A)(x,x)\sim\sum^{+\infty}_{j=0}k^{n+1-j}a_j(x)\int\chi(t)t^{n-j}dt\ \ \mbox{in $S^{n+1}_{{\rm loc\,}}(1;X)$}.
	\end{eqnarray}
	We note that \eqref{eq:ExpansionOnCircleBundleCase} was already obtained in \cite{HHMS23} for the circle bundle case and it was shown  that \(a_j=\frac{1}{2\pi}b_j\circ \pi\), \(j\geq 0\). For a more detailed  study of the relation between the coefficients \(a_1\) and \(b_1\) we refer to Remark~\ref{rmk:CoefficientsOnComplexManifolds}.  
\end{example}

\section{Preliminaries}\label{s-gue241119yyd}
\subsection{Notations}
We use the following notations throughout this article 
$\mathbb{Z}$ is the set of integers, 
$\N=\{1,2,3,\ldots\}$ is the set of natural numbers and we 
put $\N_0=\N\bigcup\{0\}$; $\mathbb R$ is the set of
real numbers. Also, $\R_+:=\{x\in\R:x>0\}$ and 
$\overline{\mathbb R}_+=\R_+\cup\{0\}$.
For a multi-index $\alpha=(\alpha_1,\ldots,\alpha_n)\in\N^n_0$ 
and $x=(x_1,\ldots,x_n)\in\mathbb R^n$, we set
\begin{equation}
\begin{split}
&x^\alpha=x_1^{\alpha_1}\ldots x^{\alpha_n}_n,\\
& \partial_{x_j}=\frac{\partial}{\partial x_j}\,,\quad
\partial^\alpha_x=\partial^{\alpha_1}_{x_1}\ldots\partial^{\alpha_n}_{x_n}
=\frac{\partial^{|\alpha|}}{\partial x^\alpha}\cdot
\end{split}
\end{equation}
Let $z=(z_1,\ldots,z_n)$, $z_j=x_{2j-1}+ix_{2j}$, $j=1,\ldots,n$, 
be coordinates on $\C^n$. We write
\begin{equation}
\begin{split}
&z^\alpha=z_1^{\alpha_1}\ldots z^{\alpha_n}_n\,,
\quad\ol z^\alpha=\ol z_1^{\alpha_1}\ldots\ol z^{\alpha_n}_n\,,\\
&\partial_{z_j}=\frac{\partial}{\partial z_j}=
\frac{1}{2}\Big(\frac{\partial}{\partial x_{2j-1}}-
i\frac{\partial}{\partial x_{2j}}\Big)\,,
\quad\partial_{\ol z_j}=\frac{\partial}{\partial\ol z_j}
=\frac{1}{2}\Big(\frac{\partial}{\partial x_{2j-1}}+
i\frac{\partial}{\partial x_{2j}}\Big),\\
&\partial^\alpha_z=\partial^{\alpha_1}_{z_1}\ldots\partial^{\alpha_n}_{z_n}
=\frac{\partial^{|\alpha|}}{\partial z^\alpha}\,,\quad
\partial^\alpha_{\ol z}=\partial^{\alpha_1}_{\ol z_1}
\ldots\partial^{\alpha_n}_{\ol z_n}
=\frac{\partial^{|\alpha|}}{\partial\ol z^\alpha}\,.
\end{split}
\end{equation}
For $j, s\in\mathbb Z$, set $\delta_{js}=1$ if $j=s$, 
$\delta_{js}=0$ if $j\neq s$.  

Let $M$ be a smooth paracompact manifold. We let $TM$ and $T^*M$ denote respectively the tangent bundle of $M$ and the cotangent bundle of $M$. The complexified tangent bundle $TM \otimes \mathbb{C}$ of $M$ will be denoted by $\mathbb C TM$, similarly we write $\mathbb C T^*M$ for the complexified cotangent bundle of $M$. Consider $\langle\,\cdot\,,\cdot\,\rangle$ to denote the pointwise
duality between $TM$ and $T^*M$; we extend $\langle\,\cdot\,,\cdot\,\rangle$ bi-linearly to $\mathbb CTM\times\mathbb C T^*M$. 
Let $Y\subset M$ be an open set and let $B$ be a vector bundle over $Y$. From now on, the spaces of distribution sections of $B$ over $Y$ and smooth sections of $B$ over $Y$ will be denoted by $\mathcal D'(Y, B)$ and $\cC^\infty(Y, B)$, respectively.
Let $\mathcal E'(Y, B)$ be the subspace of $\mathcal D'(Y, B)$ whose elements have compact support in $Y$. Let $\cC^\infty_c(Y,B):=\cC^\infty(Y,B)\cap\mathcal{E}'(Y,B)$.

\subsection{Some standard notations in microlocal and semi-classical analysis}\label{s-gue170111w} 

Let us first introduce some notions of microlocal analysis used here. Let $D\subset\R^{2n+1}$ be an open set. 
	
	\begin{definition}\label{d-gue241119ycdp}
		For any $m\in\mathbb R$, $S^m_{1,0}(D\times D\times\mathbb{R}_+)$ 
		is the space of all $s(x,y,t)\in\cC^\infty(D\times D\times\mathbb{R}_+)$ 
		such that for all compact sets $K\Subset D\times D$, all $\alpha, 
		\beta\in\N^{2n+1}_0$ and $\gamma\in\N_0$, 
		there is a constant $C_{K,\alpha,\beta,\gamma}>0$ satisfying the estimate
		\begin{equation}
		\left|\partial^\alpha_x\partial^\beta_y\partial^\gamma_t a(x,y,t)\right|\leq 
		C_{K,\alpha,\beta,\gamma}(1+|t|)^{m-|\gamma|},\ \ 
		\mbox{for all $(x,y,t)\in K\times\mathbb R_+$, $|t|\geq1$}.
		\end{equation}
		We put $S^{-\infty}(D\times D\times\mathbb{R}_+)
		:=\bigcap_{m\in\mathbb R}S^m_{1,0}(D\times D\times\mathbb{R}_+)$.
	\end{definition}
	Let $s_j\in S^{m_j}_{1,0}(D\times D\times\mathbb{R}_+)$, 
	$j=0,1,2,\ldots$ with $m_j\rightarrow-\infty$ as $j\rightarrow+\infty$. 
	By the argument of Borel construction, there always exists $s\in S^{m_0}_{1,0}(D\times D\times\mathbb{R}_+)$ 
	unique modulo $S^{-\infty}$ such that 
	\begin{equation}
	s-\sum^{\ell-1}_{j=0}s_j\in S^{m_\ell}_{1,0}(D\times D\times\mathbb{R}_+)
	\end{equation}
	for all $\ell=1,2,\ldots$. If $s$ and $s_j$ have the properties above, we write
	\begin{equation}
	s(x,y,t)\sim\sum^{+\infty}_{j=0}s_j(x,y,t)~\text{in}~ 
	S^{m_0}_{1,0}(D\times D\times\mathbb{R}_+).
	\end{equation}
	Also, we use the notation 
	\begin{equation}  
	s(x, y, t)\in S^{m}_{{\rm cl\,}}(D\times D\times\mathbb{R}_+)
	\end{equation}
	if $s(x, y, t)\in S^{m}_{1,0}(D\times D\times\mathbb{R}_+)$ and we can find $s_j(x, y)\in\cC^\infty(D\times D)$, $j\in\N_0$, such that 
	\begin{equation}
	s(x, y, t)\sim\sum^{+\infty}_{j=0}s_j(x, y)t^{m-j}\text{ in }S^{m}_{1, 0}
	(D\times D\times\mathbb{R}_+).
	\end{equation}
For smooth paracompact manifolds $M_1, M_2$,
we define the symbol spaces 
$$S^m_{{1,0}}(M_1\times M_2\times\mathbb R_+),
\quad S^m_{{\rm cl\,}}(M_1\times M_2\times\mathbb R_+)$$ 
and asymptotic sums thereof in the standard way.

Let $W_1$ be an open set in $\mathbb R^{N_1}$ and let $W_2$ be an open set in $\mathbb R^{N_2}$. Let $E$ and $F$ be vector bundles over $W_1$ and $W_2$, respectively. For any continuous operator $P: \cC^\infty_c(W_2,F)\To\mathcal{D}'(W_1,E)$, we write $P(x,y)$ to denote the distribution kernel of $P$.
A $k$-dependent continuous operator
$A_k: \cC^\infty_c(W_2,F)\To\mathcal{D}'(W_1,E)$ is called $k$-negligible on $W_1\times W_2$
if, for $k$ large enough, $A_k$ is smoothing and, for any $K\Subset W_1\times W_2$, any
multi-indices $\alpha$, $\beta$ and any $N\in\mathbb N$, there exists $C_{K,\alpha,\beta,N}>0$
such that
\[
\abs{\pr^\alpha_x\pr^\beta_yA_k(x, y)}\leq C_{K,\alpha,\beta,N}k^{-N}\:\: \text{on $K$},\ \ \forall k\gg1.
\]
In that case we write
\[A_k(x,y)=O(k^{-\infty})\:\:\text{on $W_1\times W_2$,} \quad
\text{or} \quad
A_k=O(k^{-\infty})\:\:\text{on $W_1\times W_2$.}\]
If $A_k, B_k: \cC^\infty_c(W_2, F)\To\mathcal{D}'(W_1, E)$ are $k$-dependent continuous operators,
we write $A_k= B_k+O(k^{-\infty})$ on $W_1\times W_2$ or $A_k(x,y)=B_k(x,y)+O(k^{-\infty})$ on $W_1\times W_2$ if $A_k-B_k=O(k^{-\infty})$ on $W_1\times W_2$. 
When $W=W_1=W_2$, we sometime write ``on $W$".

Let $X$ and $M$ be smooth manifolds and let $E$ and $F$ be vector bundles over $X$ and $M$, respectively. Let $A_k, B_k: \cC^\infty(M,F)\To\cC^\infty(X,E)$ be $k$-dependent smoothing operators. We write $A_k=B_k+O(k^{-\infty})$ on $X\times M$ if on every local coordinate patch $D$ of $X$ and local coordinate patch $D_1$ of $M$, $A_k=B_k+O(k^{-\infty})$ on $D\times D_1$.
When $X=M$, we sometime write on $X$.

We recall the definition of the semi-classical symbol spaces

\begin{definition} \label{d-gue140826}
Let $W$ be an open set in $\mathbb R^N$. Let
\[
S(1)=S(1;W):=\Big\{a\in\cC^\infty(W);\, \forall\alpha\in\mathbb N^N_0:
\sup_{x\in W}\abs{\pr^\alpha a(x)}<\infty\Big\}\,,\]
and let $S^0_{{\rm loc\,}}(1;W)$ be
\[\Big\{(a(\cdot,k))_{k\in\mathbb R};\,\forall\alpha\in\mathbb N^N_0,
\forall \chi\in \cC^\infty_c(W)\,:\:\sup_{k\in\mathbb R, k\geq1}\sup_{x\in W}\abs{\pr^\alpha(\chi a(x,k))}<\infty\Big\}\,.
\]
Hence $a(\cdot,k)\in S^\ell_{{\rm loc}}(1;W)$ if for every $\alpha\in\mathbb N^N_0$ and $\chi\in\cC^\infty_c(W)$, there
exists $C_\alpha>0$ independent of $k$, such that $\abs{\pr^\alpha (\chi a(\cdot,k))}\leq C_\alpha k^{\ell}$ holds on $W$.

Consider a sequence $a_j\in S^{\ell_j}_{{\rm loc\,}}(1)$, $j\in\N_0$, where $\ell_j\searrow-\infty$,
and let $a\in S^{\ell_0}_{{\rm loc\,}}(1)$. We say
\[
a(\cdot,k)\sim
\sum\limits^\infty_{j=0}a_j(\cdot,k)\:\:\text{in $S^{\ell_0}_{{\rm loc\,}}(1)$},
\]
if, for every
$N\in\N_0$, we have $a-\sum^{N}_{j=0}a_j\in S^{\ell_{N+1}}_{{\rm loc\,}}(1)$ .
For a given sequence $a_j$ as above, we can always find such an asymptotic sum
$a$, which is unique up to an element in
$S^{-\infty}_{{\rm loc\,}}(1)=S^{-\infty}_{{\rm loc\,}}(1;W):=\cap _\ell S^\ell_{{\rm loc\,}}(1)$.

Let $\ell\in\mathbb R$ and let
\[
S^\ell_{{\rm loc},{\rm cl\,}}(1):=S^\ell_{{\rm loc},{\rm cl\,}}(1;W)
\] be the set of all $a\in S^\ell_{{\rm loc}}(1;W)$ such that we can find $a_j\in\cC^\infty(W)$ independent of $k$, $j=0,1,\ldots$,  such that 
\[
a(\cdot,k)\sim
\sum\limits^\infty_{j=0}k^{\ell-j}a_j(\cdot)\:\:\text{in $S^{\ell_0}_{{\rm loc\,}}(1)$}.
\]

Similarly, we can define $S^\ell_{{\rm loc\,}}(1;Y)$ in the standard way, where $Y$ is a smooth manifold.
\end{definition}
\begin{definition}\label{def:suppt}
	Let \(U\) be an open set in \(\R^N\) and let \(F\colon U\times \R\) be a function \((x,t)\mapsto F(x,t)\). Given a subset \(I\subset \R\) we say \(\operatorname{supp}_tF(x,t)\subset I\) if for any \((x,t)\in U\times \R\) with \(t\notin I\) we have \(F(x,t)=0\).
\end{definition}
\subsection{CR geometry}\label{s-gue241119ycd}

We recall some notations concerning CR geometry. Let $X$ be a smooth orientable manifold of real dimension $2n+1,~n\geq 1$. 
We say $X$ is a (codimension one) Cauchy--Riemann (CR for short) 
manifold with CR structure \(T^{1,0}X\) if 
$T^{1,0}X\subset\mathbb{C}TX$ is a subbundle such that
\begin{enumerate}
\item[(i)] $\dim_{\mathbb{C}}T^{1,0}_{p}X=n$ for any $p\in X$.
\item[(ii)] $T^{1,0}_p X\cap T^{0,1}_p X=\{0\}$ for any $p\in X$, where 
$T^{0,1}_p X:=\overline{T^{1,0}_p X}$.
\item[(iii)] For $V_1, V_2\in \mathscr{C}^{\infty}(X,T^{1,0}X)$, 
we have $[V_1,V_2]\in\mathscr{C}^{\infty}(X,T^{1,0}X)$, where 
$[\cdot,\cdot]$ stands for the Lie bracket between vector fields. 
\end{enumerate}

From now on, we assume that $(X,T^{1,0}X)$ is a compact CR manifold of dimension $2n+1$, $n\geq1$. Since $X$ is orientable, there is a global one form $\xi\in\cC^\infty(X,T^*X)$ such that $u\lrcorner\xi=0$, for every $u\in T^{1,0}X$ and $\xi(x)\neq0$, for every $x\in X$ where \(\lrcorner\) denotes the interior product between vectors and forms. We call $\xi$ a characteristic form on $X$.

\begin{definition}\label{D:Leviform}
The Levi form $\mathcal{L}_x=\mathcal{L}^{\xi}_x$ of $X$ at $x\in X$ 
associated to a characteristic form $\xi$ is the  Hermitian form on $T^{1,0}_xX$ given by 
\begin{equation}\label{eq:2.12b}
\mathcal{L}_x:T^{1,0}_xX\times T^{1,0}_xX\to\C,
\:\: \mathcal{L}_x(U,V)=\frac{1}{2i}d\xi(U, \ol V).
\end{equation}  
\end{definition}

\begin{definition}\label{d-gue241119yydq}
A CR manifold $(X,T^{1,0}X)$ is said to be  strictly pseudoconvex if 
there exists a characteristic $1$-form $\xi$
such that for every $x\in X$ the Levi form
$\mathcal{L}^{\xi}_x$ is positive definite. 
\end{definition} 
\begin{remark}
 Given a  strictly pseudoconvex CR manifold $(X,T^{1,0}X)$ we have that \(\operatorname{Re}T^{1,0}X\) defines a contact structure on \(X\) where any characteristic $1$-form $\xi$ is a contact form for this contact structure. This follows from the observation that \(\ker \xi=\operatorname{Re}T^{1,0}X\) with \(\dim_\R \operatorname{Re}T^{1,0}X=2n\) and $\xi\wedge(d\xi)^n\neq 0$, where $\dim_\R X=2n+1$.
\end{remark}
From now on, we assume that $(X,T^{1,0}X)$ is a compact strictly pseudoconvex CR manifold of dimension $2n+1$, $n\geq1$, and we fix a characteristic one form $\xi$ on $X$ such that the associated Levi form $\mathcal{L}_x$ is positive definite at every point of $x\in X$.

There exists a unique vector field $\mathcal{T}\in\cC^\infty(X,TX)$ such that $\mathcal{T}\lrcorner\xi\equiv1$ and $\mathcal{T}\lrcorner d\xi\equiv 0$. We call $\mathcal{T}$ the Reeb vector field on $X$ with respect to $\xi$.

The Levi form $\mathcal{L}_x$ induces a Hermitian metric $\langle\,\cdot\,|\,\cdot\,\rangle$ on $\mathbb CTX$ given by 
 \begin{equation}\label{e-gue241115yyd}
 \begin{split}
&\langle\,u\,|\,v\,\rangle:=\mathcal{L}_x(u,\ol v),\ \ u, v\in T^{1,0}_xX,\\
&T^{1,0}X\perp T^{0,1}X,\\
&\mathcal{T}\perp(T^{1,0}X\oplus T^{0,1}X),\ \ \langle\,T\,|\,T\,\rangle=1.
\end{split}
 \end{equation}
We call $\langle\,\cdot\,|\,\cdot\,\rangle$ the Levi metric. Denote  $T^{*1,0}X$ and $T^{*0,1}X$ the dual bundles in 
$\C T^*X$ annihilating $\C\mathcal{T}\bigoplus T^{0,1}X$
and $\C\mathcal{T}\bigoplus T^{1,0}X$, respectively. 
Define the bundles of $(p,q)$ forms by $T^{*p,q}X:=(\Lambda^pT^{*1,0}X)\wedge(\wedge^qT^{*0,1}X)$. 
The Hermitian $\langle\,\cdot\,|\,\cdot\,\rangle$ on $\mathbb CTX$ induces a Hermitian metric 
$\langle\,\cdot\,|\,\cdot\,\rangle$ on $\oplus^n_{p,q=0}T^{*p,q}X$. For $p,q=0,1,\ldots,n$, let $\Omega^{p,q}(X):=\cC^\infty(X,T^{*p,q}X)$.

With respect to the given Hermitian metric 
$\langle\cdot|\cdot\rangle$, for $p, q\in\mathbb N_0$, $0\leq p,q\leq n$, we consider the orthogonal projection
\begin{equation}
\pi^{(p,q)}:\Lambda^{p+q}\mathbb{C}T^*X\to T^{*p,q}X.
\end{equation}
The tangential Cauchy--Riemann operator is defined to be 
\begin{equation}
\label{tangential Cauchy Riemann operator}
\overline{\partial}_b:=\pi^{(0,q+1)}\circ d:
\Omega^{0,q}(X)\to\Omega^{0,q+1}(X).
\end{equation}

Fix a volume form $dV$ on $X$ such that $\mathcal{L}_{\mathcal{T}}dV\equiv 0$, where $\mathcal{L}_{\mathcal{T}}$ denotes the Lie derivative of $\mathcal{T}$. Let $(\,\cdot\,|\,\cdot\,)$ be the $L^2$ inner product on $\Omega^{p,q}(X)$ induced by $dV$ and $\langle\,\cdot\,|\,\cdot\,\rangle$ and let $L^2_{(p,q)}(X)$ be the completion of $\Omega^{p,q}(X)$ with respect to $(\,\cdot\,|\,\cdot\,)$, $p, q\in\mathbb N_0$, $0\leq p,q\leq n$. We write $L^2(X):=L^2_{(0,0)}(X)$.
We extend $\ddbar_b$ to the $L^2$ space:
\[\ddbar_b: {\rm Dom\,}\ddbar_b\subset L^2(X)\To L^2_{(0,1)}(X),\]
where ${\rm Dom\,}\ddbar_b:=\set{u\in L^2(X);\, \ddbar_bu\in L^2_{(0,1)}(X)}$. In this work, w eassume that $X$ is embeddable, that is,  we can find a CR embedding 
\[F: X\To\mathbb C^N,\]
for some $N\in\mathbb N$. Moreover, $X$ is embeddable if and only if $\ddbar_b$ has $L^2$ closed range.  

We also recall some geometric quantities in pseudo-Hermitian geometry. The Hermitian metric $\langle\,\cdot\,|\,\cdot\,\rangle$ and the volume form $dV_\xi$ induce an $L^2$ inner product $(\,\cdot\,|\,\cdot\,)_\xi$ on $\Omega^{p,q}(X)$, $p, q\in\mathbb N_0$, $0\leq p,q\leq n$. Let 
\[\pr_b:=\pi^{1,0}\circ d: \cC^\infty(X)\To\Omega^{1,0}(X)\]
and let 
\[d_b:=\pr_b+\ddbar_b: \cC^\infty(X)\To\Omega^{0,1}(X)\oplus\Omega^{1,0}(X).\]
The \textit{sublaplacian operator} $\Delta_b:\cC^{\infty}(X)\rightarrow\cC^{\infty}(X)$ is defined by 
\begin{equation}\label{e-gue241122yyd}
(\,\Delta_b f\,|\,u\,)_\xi=\frac{1}{2}(\,d_bf\,|\,d_bu\,)_{\xi},\quad f\in\cC^\infty(X), u\in\cC^{\infty}(X).\end{equation}

\begin{definition}
    Given a contact form $\xi$ on $X,$ and $J$ a complex structure on $$HX:=\operatorname{Re}(T^{1,0}X\oplus T^{0,1}X).$$ There exists an unique affine connection $\nabla:=\nabla^{\xi,J}$ with respect to $\xi$ and $J$ such that
    \begin{enumerate}
        \item $\nabla_{v}\cC^\infty(X,HX)\subset\cC^\infty(X,HX)$ for any $v\in\cC^\infty(X,TX),$
        \item $\nabla\mathcal{T}=\nabla J=\nabla(d\xi)=0.$
        \item The torsion $\tau$ of $\nabla$ is defined by $\tau(u,u)=\nabla_u v-\nabla_vu-[u,v]$ and satisfies
        $\tau(u,v)=d\xi(u,v)\mathcal{T},$ $\tau(\mathcal{T},Jv)=-J(\mathcal{T},v)$ for any $u,v\in\cC^\infty(X,TX),$
    \end{enumerate}
\end{definition}
Let $\{L_\alpha\}_{\alpha=1}^n$ be a local frame of $T^{1,0}X$. The dual frame of $\{L_\alpha\}_{\alpha=1}^n$ is denoted by $\{\xi^{\alpha}\}_{\alpha=1}^n.$ The connection one form $\omega_\alpha^\beta$ is defined by
$$\nabla L_{\alpha}=\omega_\alpha^\beta\otimes L_{\beta}.$$
Let $L_{\overline{\alpha}}$ and $\xi^{\overline{\alpha}}$ denote $\overline{L_\alpha}$ and $\overline{\xi^{\alpha}}$ respectively, we also have $\nabla L_{\overline{\alpha}}=\omega^{\overline{\beta}}_{\overline{\alpha}}\otimes L_{\overline{\beta}}.$ The Tanaka-Webster curvature two form $\Theta_{\alpha}^{\beta}$ is defined by
\begin{equation}\label{eqn:def_tw2form}
\Theta_{\alpha}^{\beta}=d\omega_{\alpha}^{\beta}-\omega_\alpha^{\gamma}\wedge\omega_\gamma^{\beta}.
\end{equation}
It is straightforward to have the following expression for \eqref{eqn:def_tw2form}:
\begin{equation}\label{eqn:tw2form compute}
    \Theta^\beta_\alpha=R^\beta_{\alpha j \overline{k}}\xi^{j}\wedge\xi^{\overline{k}}+ A^{\alpha}_{\beta j k}\xi^{j}\wedge\xi^k+B^\alpha
_{\beta j k}\xi^{\overline{j}}\wedge\xi^{\overline{k}}+C\wedge\xi
\end{equation} for some one form $C$.
The pseudohermitian Ricci curvature  $R_{\alpha \overline{k}}$ is defined by 
$$ R_{\alpha \overline{k}}:=\sum_{\beta=j=1}^n R^\beta_{\alpha j \overline{k}}. 
$$ 
Write $d\xi=ih_{\alpha\ol{\beta}}\xi^\alpha\wedge\xi^{\ol\beta}$ and let $h^{\ol cd}$ be the inverse matrix of $h_{a\ol b}$. 
The Tanaka-Webster scalar curvature $R_{{\rm scal\,}}$ is defined by 
\begin{equation}\label{e-gue241122yydI}
R_{\rm scal}:=\sum_{\alpha,k=1}^n h^{\alpha \overline{k}}R_{\alpha\overline{k}}.
\end{equation}

\section{Proofs of Theorem~\ref{t-gue241115yyd} and Theorem~\ref{t-gue241115ycd}}\label{s-gue241120yyd}
{\tiny{{\color{white}{\subsection{ }}}}}
The main goal of this section is to prove Theorem~\ref{t-gue241115yyd}. We first generalize the result in~\cite{Hsiao_Shen_2nd_coefficient_BS_2020} to a more general setting. Let us first recall the result in~\cite{Hsiao_Shen_2nd_coefficient_BS_2020}.

\begin{theorem}\label{t-gue241120yyd}
With the notations and assumptions in Theorem~\ref{thm:ExpansionMain}, assume $dV_\xi=dV$. Recall that we work with the assumption that $X$ is embeddable. Let $(D,x)$ be any coordinate patch and let 
$\varphi:D\times D\to\C$ be any phase function satisfying 
\eqref{Eq:PhaseFuncMainThm} with $\lambda(x)\equiv1$ and \eqref{eq:PiFIO}. Suppose that 
\begin{equation}\label{e-gue241120yyd}
\mbox{$\ddbar_{b,x}\varphi(x,y)$ vainishes to infinite order on $x=y$}.
\end{equation}
Suppose further that
\begin{equation}\label{e-gue241120ycdag}
\mathcal{T}_ys(x,y,t)=0, \ \ \mathcal{T}_ys_j(x,y)=0,\ \ j=0,1,\ldots,
\end{equation}
where $s(x,y,t)$, $s_j(x,y)$, $j=0,1,\ldots$, are as in \eqref{eq:PiFIO}. Then, $s_0(x,y)$ has unique Taylor expansion at $x=y$, $s_1(x,x)$ is a global defined as a smooth function on $X$, 
\begin{equation}\label{e-gue241120ycdb}
\begin{split}
&\mbox{$s_0(x,y)-\frac{1}{2\pi^{n+1}}$ vanishes to infinite order at $x=y$, for every $x\in D$}, 
\end{split}
\end{equation}
and
\begin{equation}\label{e-gue201221yydb}
\mbox{$s_1(x,x)=\frac{1}{4\pi^{n+1}}R_{\mathrm{scal}}(x)$, for every $x\in D$}, 
\end{equation}
where $R_{{\rm scal\,}}$ is the Tanaka--Webster scalar curvature on $X$ (see \eqref{e-gue241122yydI}).
\end{theorem} 

We are going to generalize Theorem~\ref{t-gue241120yyd} to the case $dV_\xi\neq dV$. Fix $p\in X$. It is known that (see~\cite{Hsiao_Shen_2nd_coefficient_BS_2020}*{(3.99), Proposition 3.2}), there is an open local coordinate patch $D$ of $X$ with local coordinates $x=(x_1,\ldots,x_{2n+1})$, $x(p)=0$, $\mathcal{T}=\frac{\pr}{\pr x_{2n+1}}$ on $D$ and a phase $\varphi:D\times D\to\C$ satisfying 
\eqref{Eq:PhaseFuncMainThm} with $\lambda(x)=1+O(\abs{x}^3)$ and \eqref{eq:PiFIO} such that 
\begin{equation}\label{e-gue201226ycdb}
\begin{split}
&\mbox{$\ddbar_{b,x}(\varphi(x,y))$ vanishes to infinite order at $x=y$},
\end{split}
\end{equation}
and 
\begin{equation}\label{e-gue241220yydp}
\varphi(x,y)=x_{2n+1}-y_{2n+1}+\frac{i}{2}\sum_{j=1}^n\left[|z_j-w_j|^2+(\overline{z}_jw_j-z_j\overline{w}_j)\right]+O\left(|(x,y)|^4\right), 
\end{equation}
\begin{equation}\label{e-gue241220yydq}
\begin{split}
&\sum^n_{\ell,j=1}\frac{\pr^4\phi}{\pr z_j\pr z_\ell\pr\ol z_j\pr\ol z_\ell}(0,0)=
-\frac{i}{2}R_{{\rm scal\,}}(0),\\
&\sum^n_{\ell,j=1}\frac{\pr^4\phi}{\pr w_j\pr w_\ell\pr\ol w_j\pr\ol w_\ell}(0,0)=-\frac{i}{2}R_{{\rm scal\,}}(0),
\end{split}
\end{equation}
where $\frac{\pr}{\pr w_j}=\frac{1}{2}(\frac{\pr}{\pr y_{2j-1}}-i\frac{\pr}{\pr y_{2j}})$, $\frac{\pr}{\pr\ol w_j}=\frac{1}{2}(\frac{\pr}{\pr y_{2j-1}}+i\frac{\pr}{\pr y_{2j}})$, $j=1,\ldots,n$. 

Now, we work in the local coordinate patch $(D,x)$ and we assume that $\varphi$ satisfies \eqref{Eq:PhaseFuncMainThm} with $\lambda(x)=1+O(\abs{x}^3)$, \eqref{eq:PiFIO}, \eqref{e-gue201226ycdb}, \eqref{e-gue241220yydp} and \eqref{e-gue241220yydq}. 
Recall that $dV_\xi$ is the volume form induced by the contact form $\xi$ on $X$. We have the expression 
$$dV_\xi(x)=V_\xi(x)dx_1\cdots dx_{2n+1}=\frac{1}{n!}\left(\frac{d\xi}{2}\right)^n\wedge\xi,$$ $V_\xi(x)\in\cC^\infty(D)$. There exists a function $f\in\cC^\infty(X)$ satisfying $\mathcal{T}f\equiv 0$ on $X$ such that 
\begin{equation}\label{eqn:relation_two_volumes}
dV(x)=e^{2(n+1)f(x)}V_\xi(x) dx_1\cdots dx_{2n+1}=:V(x)dx_1\cdots dx_{2n+1}.
\end{equation}
We need 

\begin{lemma}
With the notations used above, for $s_0(x,y)$ in \eqref{eq:PiFIO}, we
can take $s_0(x,y)$ so that $s_0(x,y)$ is independent of $y_{2n+1}$, 

\begin{equation}\label{e-gue241128yyd}
s_0(x,y)=s_0(x,y'),\ \ \mbox{for all $(x,y)\in D\times D$},
\end{equation}
where $y'=(y_1,\ldots,y_{2n})$,
\begin{equation}\label{e-gue241120yydg}
s_0(0,0)=\frac{1}{2\pi^{n+1}}\frac{dV_\xi}{dV}(0)=\frac{1}{2\pi^{n+1}}(e^{-2(n+1)f})(0),
\end{equation}

\begin{equation}\label{e-gue241120yydu}
\begin{split}
&\frac{\pr s_0}{\pr\ol z_j}(0,0)=\frac{\pr s_0}{\pr w_j}(0,0)=0,\ \ j=1,\ldots,n,\\
&\frac{\pr s_0}{\pr z_j}(0,0)=\frac{1}{2\pi^{n+1}}\frac{\pr}{\pr z_j}(e^{-2(n+1)f})(0),\ \ j=1,\ldots,n,\\
&\frac{\pr s_0}{\pr\ol w_j}(0,0)=\frac{1}{2\pi^{n+1}}\frac{\pr}{\pr\ol w_j}(e^{-2(n+1)f})(0),\ \ j=1,\ldots,n,\\
&\frac{\pr^2s_0}{\pr\ol z_j\pr z_\ell}(0,0)=\frac{\pr^2s_0}{\pr\ol w_j\pr w_\ell}(0,0)=0,\ \ j, \ell=1,\ldots,n,
\end{split}
\end{equation}
and 
\begin{equation}\label{e-gue241125ycd}
(\mathcal{T}_xs_0)(0,0)=(\mathcal{T}_ys_0)(0,0)=0,
\end{equation}
where $\frac{\pr}{\pr w_j}=\frac{1}{2}(\frac{\pr}{\pr y_{2j-1}}-i\frac{\pr}{\pr y_{2j}})$, $\frac{\pr}{\pr\ol w_j}=\frac{1}{2}(\frac{\pr}{\pr y_{2j-1}}+i\frac{\pr}{\pr y_{2j}})$, $j=1,\ldots,n$, $f(x):=\frac{1}{2(n+1)}\log \frac{dV(x)}{dV_\xi(x)}\in\cC^\infty(X)$. 
\end{lemma}

\begin{proof}
By using integration by parts in $t$, we can take $s_0$ so that $s_0$ is independent of $y_{2n+1}$.

We assume that $s_0$ is independent of $y_{2n+1}$. Since $\varphi$ is independent of $y_{2n+1}$, $\ddbar_{b,x}\varphi(x,y)$ vanishes to infinite order at $x=y$. From this observation and $s_0$ is independent of $y_{2n+1}$, we conclude that 

\begin{equation}\label{e-gue241207yydp}
\mbox{$\ddbar_{b,x}s_0(x,y')$ vanishes to infinite order at $x=y$}.
\end{equation}
From \eqref{e-gue241207yydp}, we get 
\begin{equation}\label{e-gue241207yydq}
\frac{\pr s_0}{\pr\ol z_j}(0,0)=0,\ \ j=1,\ldots,n.
\end{equation} 

From $\ddbar_{b,x}\int e^{-it\ol\varphi(y,x)}\ol s(y,x,t)dt\equiv0$, we get 
\begin{equation}\label{e-gue241207yydr}
\mbox{$\ddbar_{b,x}(-\ol\varphi(y,x))+g(x,y)\ol\varphi(y,x)$ vanishes to infinite order at $(0,0)$},
\end{equation}
where $s(x,y,t)$ is as in \eqref{eq:PiFIO}  and $g$ is a smooth function defined near $(0,0)$. From \eqref{e-gue241220yydp} and \eqref{e-gue241207yydq}, we can check that 
\begin{equation}\label{e-gue241207yyds}
g(x,y)=O(\abs{(x,y)}^2).
\end{equation}

By using integration by parts in $t$, we get 

\begin{equation}\label{e-gue241207ycdk}
\begin{split}
&\ddbar_{b,x}\int e^{-it\ol\varphi(y,x)}\ol s(y,x,t)dt\\
&=\int e^{-it\ol\varphi(y,x)}(-(n+1)g(x,y)\ol s_0(y,x)+\ddbar_{b,x}\ol s_0(y,x))t^n+O(t^{n-1})dt.
\end{split}
\end{equation}
From \eqref{e-gue241207ycdk}, we see that 

\begin{equation}\label{e-gue241207ycdo}
\mbox{$(-(n+1)g(x,y)\ol s_0(y,x)+\ddbar_{b,x}\ol s_0(y,x))-h(x,y)\ol\varphi(y,x)$ vanishes to infinite order at $(0,0)$}, 
\end{equation}
where $h$ is a smooth function defined near $(0,0)$. From \eqref{e-gue241207yyds} and \eqref{e-gue241207ycdo}, it is not difficult to see that 
\begin{equation}\label{e-gue241207ycdg}
\frac{\pr s_0}{\pr w_j}(0,0)=\frac{\pr^2 s_0}{\pr w_j\pr\ol w_j}=0,\ \ j=1,\ldots,n.
\end{equation}

From \eqref{e-gue241207yydq} and \eqref{e-gue241207ycdg}, we get the first equation in \eqref{e-gue241120yydu}. 

From \eqref{e-gue241207yydq}, \eqref{e-gue241207ycdg} and notice that $s_0(x,x')=\frac{1}{2\pi^{n+1}}(e^{-2(n+1)f})(x)$, we get the second and third equation in \eqref{e-gue241120yydu}.

From \eqref{e-gue241207yydp} and \eqref{e-gue241207ycdg}, we get the last equation in  \eqref{e-gue241120yydu}.

Finally, from $s_0(x,x')=\frac{1}{2\pi^{n+1}}(e^{-2(n+1)f})(x)$ and $f$ is $\mathcal{T}$-invariant, we get \eqref{e-gue241125ycd}.
\end{proof}

We introduce an one-form $\widehat{\xi}=e^{2f}\xi.$ It is easy to check that $\widehat{\xi}$ is a contact form and the volume form $dV$ is induced by $\widehat{\xi},$ that is, $$dV=\frac{1}{n!}\left(\frac{d\widehat{\xi}}{2}\right)^n\wedge\widehat{\xi}.$$

With this volume, we have the following theorem:

\begin{theorem}\label{t-gue241120yydk}
    With the same notation above, the second coefficient of the Szeg\H{o} kernel is $$s_1(0,0)=\frac{1}{2\pi^{n+1}}e^{-2(n+1)f(0)}\left(\frac{1}{2}R_{{\rm scal\,}}(0)+(n+1)\Delta_b f(0)\right),$$
    where $R_{{\rm scal\,}}$ is the Tanaka-Webster scalar curvature with respect to the contact form $\xi$ (see \eqref{e-gue241122yydI}) and $\Delta_b$ denotes the CR sublaplacian operator (see \eqref{e-gue241122yyd}). 
\end{theorem}
\begin{proof}
    By replacing the volume $dV_\xi$ with $dV$ in the proof of \cite{Hsiao_Shen_2nd_coefficient_BS_2020}*{Section 3.3}, we have 
    \begin{equation}\label{eqn:Interproof 01 szego conformal}
         s_1(0,0)=-L^{(1)}_{(x,\sigma)}(s_0(0,x)s_0(x,0)e^{2(n+1)f(x)}\sigma^n)|_{(x,\sigma)=(0,1)}
    \end{equation}
    where $L^{(1)}$ is the partial differential operator in the stationary phase formula of H\"ormander. In present case, 
    \begin{equation}\label{e-gue241207ycdm}
    \begin{split}
         & L^{(1)}_{(x,\sigma)}(s_0(0,x)s_0(x,0)e^{2(n+1)f(x)}\sigma^n)\\
         &=\sum^{2}_{\mu=0}\frac{i^{-1}}{\mu!(\mu+1)!}\left(\sum^{n}_{j=1}i\frac{\partial^2}{\partial z_j\partial\overline{z_j}}+\frac{\partial^2}{\partial x_{2n+1}\partial\sigma}\right)^{\mu+1}(G^{\mu}s_0(0,x)s_0(x,0)e^{2(n+1)f(x)}\sigma^n),
         \end{split}
    \end{equation}
where $$G(x,\sigma)=F(x,\sigma)-F(0,1)-\frac{1}{2}\Big\langle \operatorname{Hess}(F)(p,1)\begin{pmatrix}
        x\\ \sigma-1
    \end{pmatrix},\begin{pmatrix}
        x\\ \sigma-1
    \end{pmatrix}\Big\rangle,$$ and $$F(x,\sigma)=\sigma\varphi(0,x)+\varphi(x,0).$$

    From \eqref{e-gue241120yydg} and \eqref{e-gue241120yydu}, we have 
   \begin{equation}\label{e-gue241121yyd}
       s_1(0,0)=\frac{-1}{2\pi^{n+1}}e^{-2(n+1)f(0)}\left(2(n+1)\sum_{i=1}^{n}\frac{\partial^2 f}{\partial z_i\partial\overline{z}_i}(0)V_{\xi}(0)+\sum_{i=1}^{n}\frac{\partial^2 V_\xi}{\partial z_i\partial\overline{z}_i}(0)+i\sum_{j,k=1}^n\frac{\partial^4\varphi}{\partial z_j\partial\overline{z}_j\partial z_k\partial\overline{z}_k}(0)\right).
       \end{equation}
       From \eqref{e-gue241220yydq}, we see that 
       \begin{equation}\label{e-gue241121yydIz}
i\sum_{j,k=1}^n\frac{\partial^4\varphi}{\partial z_j\partial\overline{z}_j\partial z_k\partial\overline{z}_k}(0)=\frac{1}{2}R_{{\rm scal\,}}(0).
       \end{equation}
Moreover, from~\cite{Hsiao_Shen_2nd_coefficient_BS_2020}*{(3.101)}, we see that 
 \begin{equation}\label{e-gue241121yydI}
 \sum_{i=1}^{n}\frac{\partial^2 V_\xi}{\partial z_i\partial\overline{z}_i}(0)=-R_{{\rm scal\,}}(0).
       \end{equation}

By definition of $\Delta_b$, it is straightforward to check that 
 \begin{equation}\label{e-gue241121yydII}
 \sum_{i=1}^{n}\frac{\partial^2 f}{\partial z_i\partial\overline{z}_i}(0)=-\frac{1}{2}(\Delta_bf)(0).
 \end{equation}

From \eqref{e-gue241121yyd}, \eqref{e-gue241121yydIz}, \eqref{e-gue241121yydI} and \eqref{e-gue241121yydII}, we get 
\begin{equation}\label{e-gue241121yydIII}
       s_1(0,0)=\frac{-1}{2\pi^{n+1}}e^{-2(n+1)f(0)}\left(-(n+1)\Delta_b f(0)-\frac{1}{2}R_{scal}\right).
   \end{equation}
   The theorem follows. 
\end{proof}

From Theorem~\ref{t-gue241120yydk}, we generalize the result in~\cite{Hsiao_Shen_2nd_coefficient_BS_2020} to any Reeb-invariant volume form. 

\begin{theorem}\label{t-gue241122yyd} 
With the notations and assumptions in Theorem~\ref{thm:ExpansionMain}, let $(D,x)$ be any coordinate patch and let 
$\varphi:D\times D\to\C$ be any phase function satisfying 
\eqref{Eq:PhaseFuncMainThm} with $\lambda(x)\equiv1$ and \eqref{eq:PiFIO}. Suppose that 
\begin{equation}\label{e-gue241120yydz}
\mbox{$\ddbar_{b,x}\varphi(x,y)$ vainishes to infinite order on $x=y$}.
\end{equation}
Suppose further that
\begin{equation}\label{e-gue241120ycdazz}
\mathcal{T}_ys(x,y,t)=0, \ \ \mathcal{T}_ys_j(x,y)=0,\ \ j=0,1,\ldots,
\end{equation}
where $s(x,y,t)$, $s_j(x,y)$, $j=0,1,\ldots$, are as in \eqref{eq:PiFIO}. Then, $s_1(x,x)$ is well-defined as a smooth function on $X$ and we have 
\begin{equation}\label{e-gue201221yydbz}
\mbox{$s_1(x,x)=\frac{-1}{2\pi^{n+1}}e^{-2(n+1)f(x)}\left(-(n+1)\Delta_b f(x)-\frac{1}{2}R_{scal}(x)\right)$,  for every $x\in D$}, 
\end{equation}
where $f(x):=\frac{1}{2(n+1)}\log \frac{dV(x)}{dV_\xi(x)}\in\cC^\infty(X)$, $R_{{\rm scal\,}}$ is the Tanaka--Webster scalar curvature on $X$ (see \eqref{e-gue241122yydI}) and $\Delta_b:\cC^{\infty}(X)\rightarrow\cC^{\infty}(X)$ denotes the CR sublaplacian operator (see \eqref{e-gue241122yyd}). 
\end{theorem}

\begin{remark}\label{rmk:CoefficientsOnComplexManifolds}
 Let \((L,h)\to M\) be a positive holomorphic line bundle with Hermitian metric \(h\) over a compact complex manifold \(M\) of \(\dim_\C M=n\). Let $R^L$ denote the curvature two form of $L$. Fix any Hermitian metric $\langle\,\cdot\,|\,\cdot\,\rangle_M$ on $\mathbb CTM$ and let $\Theta$ be the real two form on $M$ induced by $\langle\,\cdot\,|\,\cdot\,\rangle_M$. 
    Le $dV_{\Theta}$ be the volume form on $M$ induced by $\Theta$. For every $m\in N$, let \((L^m,h^m)\to M\) be the $m$-th power of $(L,h)$, where $h^m$ is the Hermitian metric on $L^m$ induced by $h$. Let $(\,\cdot\,|\,\cdot\,)_m$ be the $L^2$ inner product on $L^2(M,L^m)$ induced by $dV_{\Theta}$ and $h^m$. Let $H^0(M,L^m)$ be the space of global holomorphic sections with values in $L^m$. Let $\set{f_1,\ldots,f_{d_m}}$ be an orthonormal basis of $H^0(M,L^m)$ with respect to $(\,\cdot\,|\,\cdot\,)_m$. The Bergman kernel function is given by 
    \begin{equation}\label{e-gue241214yyd}
P_m(x):=\sum^{d_m}_{j=1}\abs{f_j(x)}^2_{h^m}\in\cC^\infty(M).
    \end{equation}
    It is well-known that (see~\cite{Catlin_BKE_1999},~\cite{HM14},~\cite{Ma_Marinescu_HMI_BKE_2007},~\cite{Zelditch_BKE_1998})
    \begin{equation}\label{e-gue241214yydI}
    \begin{split}
&P_m(x)\sim\sum^{+\infty}_{j=0}b_j(x)m^{n-j}\ \ \mbox{in $S^n_{{\rm loc\,}}(1;M)$},\\
&b_j(x)\in\cC^\infty(M),\ \ j=0,1,\ldots.
\end{split}
    \end{equation}
   Let  us consider the associated circle bundle \(\pi\colon X\to M\) as in Example~\ref{Ex:RelationToBKExpansion}. Take $\mathcal{T}\in\cC^\infty(X,TX)$ be the vector field on $X$ induced by the $S^1$ action acting on the fiber of $L^*$ and let $\xi$ be the global one form on $X$ so that $\mathcal{T}\lrcorner\xi\equiv1$ and $u\lrcorner\xi=0$, for every $u\in T^{1,0}X\oplus T^{0,1}X$. Take 
   \begin{equation}\label{e-gue241214yydIIIw}
   dV:=\frac{1}{n!}\xi\wedge\Theta^n
   \end{equation}
   and let $\Pi$ be the Szeg\H{o} projection associated to $dV$. 
   Let $s_1$ be as in Theorem~\ref{t-gue241122yyd}. We have 
   \begin{equation}\label{e-gue241214yydII}
   s_1(x,x)=\frac{1}{2\pi}b_1(\pi(x)),\ \ \mbox{for all $x\in X$},
   \end{equation}
   where $\pi: X\To M$ is the natural projection. From Theorem~\ref{t-gue241122yyd} and \eqref{e-gue241214yydII}, we get 
   \begin{equation}\label{e-gue241214ycdb}
\mbox{$b_1(\pi(x))=\frac{-(2\pi)}{2\pi^{n+1}}e^{-2(n+1)f(x)}\left(-(n+1)\Delta_b f(x)-\frac{1}{2}R_{scal}(x)\right)$, for every $x\in X$}.
   \end{equation}

We now rewrite \eqref{e-gue241214ycdb} in terms of geometric invariants on $M$. 
We can check that 
\begin{equation}\label{e-gue241214yydIIIzz}
   d\xi(x)=\sqrt{-1}R^L(\pi(x))
   \end{equation}
   and 
   \begin{equation}\label{e-gue241214yydIIIz}
dV_\xi(x)=\Bigr(\sqrt{-1}R^L(\pi(x))\Bigr)^n\frac{2^{- n}}{n!}.
   \end{equation}
   For every $x\in M$, let 
   \[\dot{R}^L(x): T^{1,0}_xM\To T^{1,0}_xM\]
   be the linear map given by  $\langle\,\dot{R}^L(x)U\,|\,V\,\rangle_M=\langle\,R^L(x)\,,\,U\wedge\ol V\,\rangle$, $U, V\in T^{1,0}_xM$.  For every $x\in M$, let ${\rm det\,}\dot{R}^L(x)=\lambda_1(x)\cdots\lambda_n(x)$, where $\lambda_j(x)$, $j=1,\ldots,n$, be the eigenvalues of $\dot{R}^L(x)$.  We can check that 
    \begin{equation}\label{e-gue241214yydIII}
   {\rm det\,}\dot{R}^L(\pi(x))=2^n\frac{dV_\xi(x)}{dV(x)},\ \ \mbox{for every $x\in X$},
   \end{equation}
and
 \begin{eqnarray}\label{e-gue241214ycd}
   f(x)&=&\frac{1}{2(n+1)}\log \frac{dV(x)}{dV_{\xi}(x)} \nonumber \\
   &=&\frac{1}{2(n+1)}\log\Bigr(\bigr( {\rm det\,}\dot{R}^L(\pi(x))\bigr)^{-1}2^{n}\Bigr),\ \ \mbox{for every $x\in X$}.
   \end{eqnarray} 

   Let $\omega:=\frac{\sqrt{-1}}{2\pi}R^L$ be the K\"ahler form on \(M\). 
The K\"ahler form $\omega$ induces a Hermitian metric $\langle\,\cdot\,|\,\cdot\,\rangle_{\omega}$ on $\mathbb CTM$ and also a Hermitian metric $\langle\,\cdot\,|\,\cdot\,\rangle_{\omega}$ on $T^{*0,q}M$ the bundle of $(0,q)$ forms on $M$, for every $q=0,1,\ldots,n$. Let $(\,\cdot\,|\,\cdot\,)_{\omega}$ be the $L^2$ inner product on $\oplus^n_{q=0}\Omega^{0,q}(M)$ induced by $\langle\,\cdot\,|\,\cdot\,\rangle_{\omega}$, where $\Omega^{0,q}(M):=\cC^\infty(M,T^{*o,q}M)$. The complex Laplacian with respect to $\omega$ is the operator $\Delta_\omega:\cC^{\infty}(M)\rightarrow\cC^{\infty}(M)$ given by 
\begin{equation}\label{e-gue241122yydz}
(\,\Delta_\omega f\,|\,u\,)_\omega=(\,df\,|\,du\,)_{\omega},\quad f\in\cC^\infty(M), u\in\cC^{\infty}(M).\end{equation}
Let $g\in\cC^\infty(M)$ and let $\hat g\in\cC^\infty(X)$ given by $\widehat g(x):=g(\pi(x))$, for all $x\in X$. We can check that
\begin{equation}\label{e-gue241214ycdqq}
\Delta_b\widehat{g}(x)=\frac{1}{2\pi}(\Delta_\omega g)(\pi(x)),\ \ \mbox{for all $x\in X$}.
\end{equation}

Let $r$ be the scalar curvature on $M$ with respect to $\omega=\frac{\sqrt{-1}}{2\pi}R^L$ (see~\cite{Hsiao_BT_coefficient_2012}*{(1.8)}). 
It was shown in~\cite{Herrmann_Hsiao_Li_Q-R-Sasakian_Szego_2018}*{Theorem 3.5} that \begin{equation}\label{e-gue241214ycdq}
r(\pi(x))=4\pi R_{\rm{scal}}(x),\ \ \mbox{for every $x\in X$}.
\end{equation}

From \eqref{e-gue241214ycdb}, \eqref{e-gue241214ycd}, \eqref{e-gue241214ycdqq} and \eqref{e-gue241214ycdq}, we get 

\begin{equation}\label{e-gue241214ycdl}
b_1(\pi(x))=(2\pi)^{-n}{\rm det\,}\dot{R}^L(\pi(x))\Bigr(\frac{1}{8\pi}r(\pi(x))
+(n+1)(\Delta_bf)(x)\Bigr).
\end{equation}
From \eqref{e-gue241214ycd} and \eqref{e-gue241214ycdqq}, we can check that 
\begin{equation}\label{e-gue241214ycdj}
(n+1)(\Delta_bf)(x)=\frac{1}{4\pi}\Bigr(\hat r(\pi(x))-r(\pi(x)\Bigr),
\end{equation}
for every $x\in X$, where $\hat r$ is given by~\cite{Hsiao_BT_coefficient_2012}*{(1.8)}. From \eqref{e-gue241214ycdl} and \eqref{e-gue241214ycdj}, we get 
\begin{equation}\label{e-gue241214ycdk}
b_1(x)=(2\pi)^{-n}{\rm det\,}\dot{R}^L(x)\Bigr(\frac{1}{4\pi}\hat r(x)-\frac{1}{8\pi}r(x)\Bigr),
\end{equation}
for every $x\in M$. The formula \eqref{e-gue241214ycdk} coincides with the result in~\cite{Hsiao_BT_coefficient_2012}*{Theorem 1.4}. 

The computation of the coefficients of Bergman kernel in more general setting can be found in \cite{Ma_Marinescu_JRAM_2012}. We also refer the reader to~\cite{Herrmann_Hsiao_Li_Q-R-Sasakian_Szego_2018} for the computation of the coefficients  \(b_0\), \(b_1\) and \(b_2\) for quasi-regular Sasakian manifolds.
\end{remark}

We need the following partial refinement of the main result obtained in~\cite{HHMS23}. The refinement of the main result obtained in~\cite{HHMS23} is basically given by equation \eqref{eqn:ThmI1InProof} below.   

\begin{theorem}\label{t-gue241123yyd}
	Let $(X,T^{1,0}X)$ be a compact orientable strictly pseudoconvex embeddable CR manifold of dimension $2n+1$, \(n\geq 1\) with volume form \(dV\). Fix a global one form $\xi\in\cC^\infty(X,T^*X)$ such that for any \(x\in X\) we have $\xi(x)\neq0$,  $\ker\xi(x)=\operatorname{Re}T_x^{1,0}X$ and the respective Levi form $\mathcal{L}_x$ is positive definite. Let $(D,x)$ be any coordinate patch and let 
	$\varphi:D\times D\to\C$ be any phase function satisfying \eqref{Eq:PhaseFuncMainThm} and \eqref{eq:PiFIO}. For any formally self-adjoint first order classical pseudodifferential operator \(P\) on \(X\) whose principle symbol  at \(\xi\) is positive there exist smooth functions
	\begin{eqnarray*}
		a_{j,s}(x,y)\in\cC^\infty(D\times D),\ \ j,s=0,1,\ldots,
	\end{eqnarray*}
	such that for any \(\chi\in \mathscr{C}^\infty_c(\R_+)\) putting \(\chi_k(t):=\chi(k^{-1}t)\), \(k,t\in\R_+\), one has for the distributional kernel $\chi_k(A)(x,y)$ of $\chi_k(A)$ where \(A:=\Pi P \Pi\colon \operatorname{Dom}(A)\to L^2(X)  \) that
	\begin{equation*}
		\chi_k(A)(x,y)=\int_0^{+\infty} 
		e^{ikt\varphi(x,y)}b^\chi(x,y,t,k)dt+O\left(k^{-\infty}\right)~\text{on}~D\times D,
	\end{equation*}
	 with $b^\chi(x,y,t,k)\in S^{n+1}_{\mathrm{loc}}
	 (1;D\times D\times{\R}_+)$,
	 \begin{equation*}
	 	\begin{split}
	 		&b^\chi(x,y,t,k)\sim\sum_{j=0}^{+\infty}b^\chi_{j}(x,y,t)k^{n+1-j}~
	 		\text{in $S^{n+1}_{\mathrm{loc}}(1;D\times D\times{\R}_+)$},\\
	 		&b^\chi_j(x,y,t)\in\mathscr{C}^\infty(D\times D\times{\R}_+),~j=0,1,2,\ldots,
	 	\end{split}
	 \end{equation*}
	 so that 
	 \begin{equation}\label{eqn:ThmI1InProof}
	 	b^\chi_j(x,y,t)=\sum^{+\infty}_{s=0}a_{j,s}(x,y)\chi^{(s)}(t)t^{n+s-j},\quad j=0,1,\ldots,
	 \end{equation}
	 where \(\chi^{(s)}:=(\partial/\partial t)^s\chi\) and the infinite sum in \eqref{eqn:ThmI1InProof} converges uniformly in $\cC^{\infty}(K\times K\times I)$ topology, for any compact subsets $K\subset D$, $I\subset\mathbb R$, for all $j=0,1\dots$.  
\end{theorem} 

\begin{proof}
From~\cite{HHMS23}*{Lemma 4.2, Lemma 4.3, Theorem 4.11}, we see that there is a phase $\varphi_1(x,y)\in\cC^\infty(D\times D)$  satisfying 
\eqref{Eq:PhaseFuncMainThm} and \eqref{eq:PiFIO} such that 
\begin{equation}\label{e-gue241124ycd}
\chi_k(A)(x,y)=\int_0^{+\infty} 
e^{ikt\varphi_1(x,y)}\hat b^\chi(x,y,t,k)dt+O\left(k^{-\infty}\right)~\text{on}~D\times D,
\end{equation}
where $\hat b^\chi(x,y,t,k)\in S^{n+1}_{\mathrm{loc}}
(1;D\times D\times{\R}_+)$,
\begin{equation}\label{e-gue241124ycdI}
\begin{split}
&\hat b^\chi(x,y,t,k)\sim\sum_{j=0}^{+\infty}\hat b^\chi_{j}(x,y,t)k^{n+1-j}~
\text{in $S^{n+1}_{\mathrm{loc}}(1;D\times D\times{\R}_+)$},\\
&\hat b^\chi_j(x,y,t)\in\mathscr{C}^\infty(D\times D\times{\R}_+),~j=0,1,2,\ldots,
\end{split}
\end{equation}
and for some bounded open interval $I^\chi\Subset\R_+$,
\begin{equation}\label{e-gue241115ycdaz}
\begin{split}
{\rm supp\,}_t\hat b^\chi(x,y,t,k),~{\rm supp\,}_t \hat b^\chi_j(x,y,t)\subset I^\chi,\ \ j=0,1,2,\ldots~.
\end{split}
\end{equation}
and 
\begin{equation}\label{e-gue241124yyd}
\begin{split}
&\hat b^\chi_j(x,y,t)=\sum^{N_j}_{s=0}\hat a_{j,s}(x,y)\chi^{(s)}(t)t^{n+s-j},\ \ j=0,1,\ldots,\\
&\hat a_{j,s}(x,y)\in\cC^\infty(D\times D),\ \ s=0,1,\ldots,N_j,\ \ j=0,1,\ldots,
\end{split}
\end{equation}
where $N_j\in\mathbb N$, $j=0,1,\ldots$ and \(\hat a_{j,s}(x,y)\), \( s=0,1,\ldots,N_j,\,  j=0,1,\ldots\), are independent of \(\chi\). 

Let $\varphi\in\cC^\infty(D\times D)$
 be any phase function satisfying 
\eqref{Eq:PhaseFuncMainThm} and \eqref{eq:PiFIO}. From~\cite{Hsiao_Marinescu_Szego_lower_energy_2017}*{Theorem 5.4}, we see that 
\begin{equation}\label{e-gue241124yydI}
\varphi(x,y)=g(x,y)\varphi_1(x,y)+f(x,y),
\end{equation}
 where $g(x,y)$, $f(x,y)\in\cC^\infty(D\times D)$, $f(x,y)$ vanishes to infinite order at $x=y$. Fix $p\in D$ and take local coordinates $x=(x_1,\ldots,x_{2n+1})$ defined on $D$ so that $x(p)=0$ and $\mathcal{T}=\frac{\pr}{\pr x_{2n+1}}$ where \(\mathcal{T}\) is the Reeb vector field with respect to \(\xi\). In order to verify \eqref{eqn:ThmI1InProof} it is enough to consider the statement for this choice of coordinates.
  From Malgrange preparation theorem, near $(0,0)$, we have 
 \begin{equation}\label{e-gue241124yydII}
\begin{split}
&\varphi(x,y)=\alpha(x,y)(-y_{2n+1}+\hat\varphi(x,y')),\\
&\varphi_1(x,y)=\beta(x,y)(-y_{2n+1}+\hat\varphi_1(x,y')),\\
\end{split}
 \end{equation}
where $y'=(y_1,\ldots,y_{2n})$, $\alpha, \beta, \hat\varphi, \hat\varphi_1$ are smooth functions defiend near $(0,0)$. We may assume that \eqref{e-gue241124yydII} hold on $D\times D$ and hence $\alpha, \beta, \hat\varphi, \hat\varphi_1\in\cC^\infty(D\times D)$. We may take $D$ small enough so that $\alpha(x,y)\neq0$, $\beta(x,y)\neq0$ at every point of $D\times D$. It is not difficult to see that 
\begin{equation}\label{e-gue241124yyda}
\begin{split}
&{\rm Im\,}\alpha(x,y)=O(\abs{(x-y}),\\
&{\rm Im\,}\beta(x,y)=O(\abs{(x-y}).
\end{split}
\end{equation} 

From \eqref{e-gue241124yydI} and \eqref{e-gue241124yydII}, it is not difficult to see that 
\begin{equation}\label{e-gue241124yydr}
\hat\varphi(x,y')-\hat\varphi_1(x,y')=O(\abs{(x'-y')}^N),\ \ \mbox{for all $N\in\mathbb N$}. 
\end{equation}
In order to simplify notation we will use \(\hat b(x,y,t,k)\), \(\hat b_j(x,y,t)\) and \(I\) instead of \(\hat b^\chi(x,y,t,k)\) and \(\hat b^\chi_j(x,y,t)\) and \(I^\chi\) keeping in mind that \(\hat b(x,y,t,k)\), \(\hat b_j(x,y,t)\) and \(I\) will depend on \(\chi\).
 Let $\Td{\hat b}(x,y,\Td t,k)\in\cC^\infty(D\times D\times I^{\mathbb C})$ be an almost analytic extension of $\hat b(x,y,t,k)$ in $t$ variable, where $I^{\mathbb C}$ is a bounded open set in $\mathbb C$ so that $I^{\mathbb C}\cap\mathbb R=I$. Recall that $I$ is a bounded open interval as in \eqref{e-gue241115ycdaz}.
 We take $\Td{\hat b}(x,y,\Td t,k)$ so that 
 \begin{equation}\label{e-gue241115ycdazz}
\begin{split}
{\rm supp\,}_{\Td t}\Td{\hat b}(x,y,\Td t,k)\subset I^{\mathbb C},\ \ j=0,1,2,\ldots, 
\end{split}
\end{equation}
and 
\begin{equation}\label{e-gue241124yydm}
\Td{\hat b}(x,y,\frac{\alpha(x,y)}{\beta(x,y)}t,k)\sim\sum^{+\infty}_{j=0}\Td{\hat b}_j(x,y,\frac{\alpha(x,y)}{\beta(x,y)}t)k^{n+1-j}\ \ \mbox{in $S^{n+1}_{{\rm loc\,}}(1;D\times D\times\mathbb R)$},
\end{equation}
 where $\Td{\hat b}_j(x,y,\Td t)\in\cC^\infty(D\times D\times I^{\mathbb C})$ denotes an almost analytic extension of $\hat b_j(x,y,t)$ in $t$ variable, $j=0,1,\ldots$, so that ${\rm supp\,}_{\Td t}\Td{\hat b}_j(x,y,\Td t)\subset I^{\mathbb C}$, $j=0,1,2,\ldots$, and 
 \begin{equation}\label{e-gue241124yydn}
\Td{\hat b}_j(x,y,\Td t)=\sum^{N_j}_{s=0}\hat a_{j,s}(x,y)(\frac{\pr^s}{\pr\Td t^s}\Td\chi)(\frac{\alpha(x,y)}{\beta(x,y)}t)(\frac{\alpha(x,y)}{\beta(x,y)}t)^{n+s-j}, 
 \end{equation}
 $\Td\chi$ is an almost analytic extension of $\chi$ with ${\rm supp\,}\Td\chi\subset I^{\mathbb C}$ and the \(\hat a_{j,s}(x,y)\)'s are independent of \(\chi\).

 From \eqref{e-gue241124yyda}, we can apply Stokes' theorem to \eqref{e-gue241124ycd} and get 
\begin{equation}\label{e-gue241124ycdp}
\chi_k(A)(x,y)=\int_0^{+\infty} 
e^{ikt\alpha(x,y)(-y_{2n+1}+\hat\varphi_1(x,y'))}\Td{\hat b}(x,y,\frac{\alpha(x,y)}{\beta(x,y)}t,k)\frac{\alpha(x,y)}{\beta(x,y)}dt+O\left(k^{-\infty}\right)~\text{on}~D\times D.
\end{equation}
 From \eqref{e-gue241124yydr} and notice that ${\rm Im\,}\hat\varphi_1(x,y')\geq C\abs{x'-y'}^2$, where $C>0$ is a constant, we can change $\hat\varphi_1(x,y')$ in \eqref{e-gue241124ycdp} to $\hat\varphi(x,y')$ and get 
 \begin{equation}\label{e-gue241124ycdq}
\chi_k(A)(x,y)=\int_0^{+\infty} 
e^{ikt\varphi(x,y))}\Td{\hat b}(x,y,\frac{\alpha(x,y)}{\beta(x,y)}t,k)\frac{\alpha(x,y)}{\beta(x,y)}dt+O\left(k^{-\infty}\right)~\text{on}~D\times D.
\end{equation}
 From Borel construction, we can find $ b^\chi(x,y,t,k)\in S^{n+1}_{\mathrm{loc}}
(1;D\times D\times{\R}_+)$, such that 
$b^\chi(x,y,t,k)-\frac{\alpha(x,y)}{\beta(x,y)}\Td{\hat b}(x,y,\frac{\alpha(x,y)}{\beta(x,y)}t,k)$ vanishes to infinite order at $x=y$, 
\begin{equation}\label{e-gue241124ycdIg}
\begin{split}
&b^\chi(x,y,t,k)\sim\sum_{j=0}^{+\infty}b^\chi_{j}(x,y,t)k^{n+1-j}~
\text{in $S^{n+1}_{\mathrm{loc}}(1;D\times D\times{\R}_+)$},\\
&b^\chi_j(x,y,t)\in\mathscr{C}^\infty(D\times D\times{\R}_+),~j=0,1,2,\ldots,
\end{split}
\end{equation}
and ${\rm supp\,}_t b^\chi(x,y,t,k)\subset I^\chi$, ${\rm supp\,}_t b^\chi_j(x,y,t)\subset I^\chi$, $j=0,1,2,\ldots$, and \begin{equation}\label{e-gue241124ycdg}
\begin{split}
    &b^\chi_j(x,y,t)=\sum^{+\infty}_{s=0}a_{j,s}(x,y)\chi^{(s)}(t)t^{n+s-j},\quad j=0,1,\ldots,\\
     &a_{j,s}(x,y)\in\cC^\infty(D\times D),\ \ j,s=0,1,\ldots,
\end{split}\end{equation}
where the infinite sum in \eqref{e-gue241124ycdg} converges uniformly in $\cC^{\infty}(K\times K\times I)$ topology, for any compact subsets $K\subset D$, \(I\subset \R\) and the \(a_{j,s}(x,y)\)'s are independent of \(\chi\). Since $b^\chi(x,y,t,k)-\frac{\alpha(x,y)}{\beta(x,y)}\Td{\hat b}(x,y,\frac{\alpha(x,y)}{\beta(x,y)}t,k)$ vanishes to infinite order at $x=y$, we can replace $\frac{\alpha(x,y)}{\beta(x,y)}\Td{\hat b}(x,y,\frac{\alpha(x,y)}{\beta(x,y)}t,k)$ in \eqref{e-gue241124ycdq} by $b^\chi(x,y,t,k)$. The theorem follows. 
 \end{proof}

Fix $p\in D$, we take local coordinates $x=(x_1,\ldots,x_{2n+1})$ defined on $D$ so that $x(p)=0$ and $\mathcal{T}=\frac{\pr}{\pr x_{2n+1}}$. Until further notice, we work in $D$ and work with the local coordinates $x=(x_1,\ldots,x_{2n+1})$.
In the rest of this section, we assume that $\varphi$ satisfies \eqref{Eq:PhaseFuncMainThm} with $\lambda(x)=1+O(\abs{x}^3)$, \eqref{eq:PiFIO}, \eqref{e-gue201226ycdb}, \eqref{e-gue241220yydq} and has the following form 
\begin{equation}\label{eqn:ThmI2}
\begin{split}
    &\varphi(x,y)=-y_{2n+1}+x_{2n+1}+\hat\varphi(x,y'),\quad y'=(y_1,\cdots,y_{2n}),\\
    &\hat\varphi(x,y')\in\cC^\infty(D\times D),\\
    &\hat\varphi(x,y')=\frac{i}{2}\sum_{j=1}^n\left[|z_j-w_j|^2+(\overline{z}_jw_j-z_j\overline{w}_j)\right]+O\left(|(x,y')|^4\right).
    \end{split}
\end{equation}
This is always possible (see~\cite{Hsiao_Shen_2nd_coefficient_BS_2020}*{(3.99), Proposition 3.2}).  
By using integration by parts, we can take $b^\chi_j(x,y,t)$ in \eqref{Eq:LeadingTermMainThm}, $j=0,1,\ldots$, so that 
\begin{equation} \label{e-gue241128yydz}
\begin{split}
&b^\chi_j(x,y,t)=b^\chi_j(x,y',t)=\sum^{+\infty}_{s=0}a_{j,s}(x,y')\chi^{(s)}(t)t^{n+s-j},\quad j=0,1,\ldots,\\
     &a_{j,s}(x,y')\in\cC^\infty(D\times D),\ \ j,s=0,1,\ldots
     \end{split}
     \end{equation}
and the \(a_j(x,y')\)'s are independent of \(\chi\).
Assume $D=\widehat{D}\times(-\epsilon,\epsilon),$ $\widehat{D}$ is an open set of $\R^{2n}$ of $0\in R^{2n}$, $\varepsilon>0$ is a constant.
Fix $\tau\in\cC^{\infty}_c((-\epsilon,\epsilon)), \tau\equiv 1$ near $0\in\mathbb R$ and $m_0\in\mathbb R$. Consider the map 
\begin{align*}
    \reallywidetilde{\chi_k(A)}:\cC^{\infty}_{c}(\hat D)&\longrightarrow\cC^{\infty}(D),\\
    u&\longmapsto\frac{k}{2\pi}\int\chi_k(A)(x,y)e^{ikm_0y_{2n+1}}\tau(y_{2n+1})u(y')dy'dy_{2n+1}.
\end{align*}
This distribution kernel of $\reallywidetilde{\chi_k(A)}$ is of the form:
\begin{equation}\label{eqn:ThmI3}
\begin{split}
    \reallywidetilde{\chi_k(A)}(x,y')&=\frac{k}{2\pi}\int e^{ikt\varphi(x,y)+ikm_0y_{2n+1}}b^\chi(x,y',t,k)\tau(y_{2n+1})dy_{2n+1}dt+O(k^{-\infty})\\
    &=\frac{k}{2\pi}\int e^{ikt(-y_{2n+1}+\widehat{\varphi}(x,y'))+ikm_0y_{2n+1}}b^\chi(x,y',t,k)\tau(y_{2n+1})dy_{2n+1}dt+O(k^{-\infty})\\
    &=\frac{k}{2\pi}\int e^{ikt(-y_{2n+1}+\widehat{\varphi}(x,y'))+ikm_0y_{2n+1}}b^\chi(x,y',t,k)dy_{2n+1}dt+O(k^{-\infty})\\
    &=e^{ikm_0\widehat{\varphi}(x,y')}b^\chi(x,y',m_0,k)+O(k^{-\infty}). \end{split}
\end{equation} 

We first prove the following uniqueness of the expansion:

\begin{theorem}\label{t-gue241125yyda}
    With the notations and assumptions above, assume
    $$\int e^{it\varphi(x,y)}a(x,y',t,k)dt=\int e^{it\varphi(x,y)}\widehat{a}(x,y',t,k)dt+O(k^{-\infty})\quad \text{on }D\times D, $$
    where $\varphi$ satisfies \eqref{eqn:ThmI2}, $a(x,y',t,k)$, $\hat a(x,y',t,k)\in S^{n+1}_{{\rm loc\,}}(1;D\times D\times\mathbb R)$, 
    \[\begin{split}
        &\mbox{$a(x,y',t,k)\sim\sum^{\infty}_{j=0}a_j(x,y',t)k^{n+1-j}$ in $S^{n+1}_{{\rm loc}}(1;D\times D\times\mathbb R)$},\\
        &\mbox{$\widehat{a}(x,y',t,k)\sim\sum^{\infty}_{j=0}\widehat{a}_j(x,y',t)k^{n+1-j}$ in $S^{n+1}_{{\rm loc\,}}(1;D\times D\times\mathbb R)$,}\\
        &\mbox{$a_j, \hat a_j\in\cC^\infty(D\times D\times\mathbb R)$, $j=0,1,\ldots$},\\
        &\mbox{$\underset{t}{\rm supp\,}a\subset I$, $\underset{t}{\rm supp\,}\widehat{a}\subset I$, $\underset{t}{\rm supp\,}a_j\subset I,$ $\underset{t}{\rm supp\,}\widehat{a}_j\subset I, j=0,1,\dots,$}
        \end{split}\]
        $I \Subset\mathbb R$ is a bounded interval. Then
    $$a_j(x,y',t)=\widehat{a}_j(x,y',t)+O(|(x'-y')|^N),\quad \forall N,\forall t\in I, \forall j=0,1,\dots.$$
\end{theorem}

\begin{proof}
    Fix $m_0\in I.$ We can repeat the process in \eqref{eqn:ThmI3} and deduce that 
    $$e^{ikm_0\widehat{\varphi}(x,y')}a(x,y',m_0,k)=e^{ikm_0\widehat{\varphi}(x,y')}\widehat{a}(x,y',m_0,k)+F_k(x,y'),\quad F_k(x,y')=O(k^{-\infty}).$$
    Hence, 
    \begin{equation}\label{eqn:ThmI4}
        a(x,y',m_0,k)-\widehat{a}(x,y',m_0,k)=e^{-ikm_0\widehat{\varphi}(x,y')}F_k(x,y').
    \end{equation}
    From \eqref{eqn:ThmI4}, we have 
    \begin{align*}
        a_0(x,x',m_0)-\widehat{a}_0(x,x',m_0)&=\lim_{k\to\infty}\frac{1}{k^{n+1}}e^{-ikm_0\widehat{\varphi}(x,x')}F_k(x,x')\\
        &=\lim_{k\to\infty}\frac{1}{k^{n+1}}e^{ikm_0x_{2n+1}}F_k(x,x')=0.
    \end{align*}
    Similarly,
    $$\left(\partial^{\alpha}_{x}\partial^{\beta}_{y'}(a_0-\widehat{a}_0)\right)(x,x',m_0)=\lim_{k\to\infty}\frac{1}{k^{n+1}}\partial^{\alpha}_{x}\partial^{\beta}_{y'}\left(e^{-ikm_0\widehat{\varphi}(x,y')}F_k(x,y')|_{(x,x')}\right)=0,$$
    for all $\alpha\in\N_0^{2n+1},$ $\beta\in\N_0^{2n}.$
    Similarly, we can show that $a_j(x,y',t)=\widehat{a}_j(x,y',t)+O(|x'-y'|^N),$ for all $N\in\N,$ $j=1,2,\cdots.$
\end{proof}

From now on, we assume that $b^\chi_j(x,y,t)$, $j=0,1,\ldots,n$, satisfy \eqref{e-gue241128yydz}. 

\begin{theorem}\label{thm:s1_indep_of_deri_of_chi}
    With the notations and assumptions used above, recall that $\varphi$ satisfies \eqref{e-gue201226ycdb}, \eqref{e-gue241220yydq}, \eqref{eqn:ThmI2}, and  $b^\chi_j(x,y,t)$, $j=0,1,\ldots$, in \eqref{Eq:LeadingTermMainThm}, satisfy \eqref{e-gue241128yydz}.
    Then, for all $x\in D$, $t\in\mathbb R$ and \(\chi\in\mathscr{C}_c^\infty(\R_+)\), 
    \begin{equation}\label{eqn:ThmIII4}
    \begin{split}
    &b^\chi_0(0,0,t)=s_0(0,0)\chi(t)t^n,\\
&\frac{\pr b^\chi_0}{\pr x_j}(0,0,t)=\frac{\pr s_0}{\pr x_j}(0,0)\chi(t)t^n,\ \ j=1,\ldots,2n+1,\\
&\frac{\pr b^\chi_0}{\pr y_j}(0,0,t)=\frac{\pr s_0}{\pr y_j}(0,0)\chi(t)t^n,\ \ j=1,\ldots,2n+1,\\
&\frac{\pr^2b^\chi_0}{\pr\ol z_j\pr z_\ell}(0,0,t)=\frac{\pr^2b^\chi_0}{\pr\ol w_j\pr w_\ell}(0,0,t)=0,\ \ j, \ell=1,\ldots,n,
\end{split}
    \end{equation}
and
    \begin{equation}\label{eqn:ThmIII5}
        a_{1,s}(p,p)=0,\quad s=1,2,\cdots,
    \end{equation}
    where $s_0(x,y)$ is as in \eqref{eq:PiFIO} 
    and $s_0(x,y)$ satisfies \eqref{e-gue241128yyd}, 
    \eqref{e-gue241120yydg}, \eqref{e-gue241120yydu}, \eqref{e-gue241125ycd} and $a_{1,s}(x,y')$ is as in \eqref{e-gue241128yydz}, $s=1,2,\cdots.$
\end{theorem}

\begin{proof}
The prove of \eqref{eqn:ThmIII4} is the same as the proofs of \eqref{e-gue241120yydu} and \eqref{e-gue241125ycd}. 
   
   From \eqref{eq:PiFIO} and \eqref{eqn:ThmI2}, we can check that  
    \[\begin{split}
        &\left((-i)\mathcal{T}\Pi\right)(x,y)=\int e^{it\varphi(x,y)}g(x,y',t)dt,\\
        &\left(\Pi(-i\mathcal{T})\right)(x,y)=\int e^{it\varphi(x,y)}\hat g(x,y',t)dt,\\
        &g(x,y',t)\sim\sum^{\infty}_{j=0}t^{n+1-j}g_j(x,y')\ \ \mbox{in $S^{n+1}_{1,0}(D\times D\times\mathbb R)$},\\
        &\hat g(x,y',t)\sim\sum^{\infty}_{j=0}t^{n+1-j}\hat g_j(x,y')\ \ \mbox{in $S^{n+1}_{1,0}(D\times D\times\mathbb R)$},\\
        &g_j(x,y'), \hat g_j(x,y')\in\cC^\infty(D\times D),\ \ j=0,1,\ldots,
        \end{split}\]
    \begin{equation}\label{eqn:ThmIII7}
    g_0(x,y')-\hat g_0(x,y')=O(\abs{x-y}^3).
    \end{equation}
    Let $\reallywidetilde{\chi_k(A)}$(x,y') be as in \eqref{eqn:ThmI3}. We have 
\begin{align}
        \frac{1}{k}(-i\mathcal{T})\Pi\reallywidetilde{\chi_k(A)}(p,p)=k^{n+1}m_0b^\chi_0(p,p',m_0)+k^nm_0b^\chi_1(p,p',m_0)+O(k^{n+1})\nonumber\\=k^{n+1}m_0^{n+1}\chi(m_0)s_0(p,p')+\sum^{\infty}_{s=0}k^na_{1,s}(p,p')\chi^{(s)}(m_0)m_0^{n+s}+O(k^{n-1}). \label{eqn:ThmIII10}
    \end{align}
Let $\widehat{\chi}(t)=t\chi(t).$ We have $\frac{1}{k}\Pi(-i\mathcal{T})\reallywidetilde{\chi_k(A)}=\reallywidetilde{\widehat{\chi_k}(A)}.$ From this observation, we get 
    \begin{equation}\label{eqn:ThmIII11}
    \begin{split}
        &\reallywidetilde{\widehat{\chi_k}(A)}(p,p)\\
&=k^{n+1}m_0^{n+1}\chi(m_0)s_0(p,p')+\sum^{\infty}_{s=0}k^na_{1,s}(p,p')(t\chi(t))^{(s)}|_{t=m_0}m_0^{n-1+s}+O(k^{n-1}).
    \end{split}
    \end{equation}

From \eqref{eqn:ThmIII7} and stationary phase formula of H\"ormander, we can check that 
    \begin{equation}\label{eqn:ThmIII9}
        \left(\frac{1}{k}(-i\mathcal{T})\Pi\reallywidetilde{\chi_k(A)}\right)(p,p)-\left(\frac{1}{k}\Pi(-i\mathcal{T})\reallywidetilde{\chi_k(A)}\right)(p,p)=O(k^{n-1}).
    \end{equation}
    From \eqref{eqn:ThmIII10}, \eqref{eqn:ThmIII11} and \eqref{eqn:ThmIII9}, we get 
    \begin{equation}\label{eqn:ThmIII12}
        \sum^{\infty}_{s=0}a_{1,s}(p,p')\chi^{(s)}(m_0)m_0^{n+s}=\sum^{\infty}_{s=0}a_{1,s}(p,p')(t\chi(t))^{(s)}|_{t=m_0}m_0^{n-1+s}.
    \end{equation}

    Take $\chi$ so that $\chi(m_0)\neq 0$ and $\chi^{(s)}(m_0)=0,$ for all $s\geq 1.$ From \eqref{eqn:ThmIII12}, we get 
    $$a_{1,0}(p,p')\chi(m_0)m_0^n=a_{1,0}(p,p')\chi(m_0)m_0^n+a_{1,1}(p,p')\chi(m_0)m_0^n.$$ Thus, $a_{1,1}(p,p')=0.$
    Suppose $a_{1,s}(p,p')=0,$ for all $1\leq s\leq N_0,$ for some $N_0\in\N.$ Take $\chi$ so that $\chi^{(N_0)}(m_0)\neq 0$, $\chi^{(s)}(m_0)=0$, $s>N_0.$ By induction assumption and \eqref{eqn:ThmIII12}, we get 
    $$a_{1,0}(p,p')\chi(m_0)m_0^n=a_{1,0}(p,p)\chi(m_0)m_0^n+a_{1,N_0+1}(p,p')\chi^{(N_0)}(m_0)m_0^{n+N_0}.$$
    Thus, $a_{1,N_0+1}(p,p')=0.$ By induction we get \eqref{eqn:ThmIII5}.
\end{proof}

\begin{proof}[Proof of Theorem~\ref{t-gue241115yyd}]
Let $\chi, \tau\in\cC^\infty_c(\mathbb R_+)$, $\tau\equiv1$ on ${\rm supp\,}\chi$. By Theorem~\ref{t-gue241123yyd}, Theorem~\ref{thm:s1_indep_of_deri_of_chi}, we have the following expansion:
\begin{equation}\label{eqn:chiktp}
\begin{split}
    &\chi_k(T_P)(x,y)=\int e^{ikt\varphi(x,y)}\left(k^{n+1}\hat b_0(x,y',t,\chi)+k^n\hat b_1(x,y',t,\chi)+O(k^{n-1})\right)dt,\\
    &\tau_k(T_P)(x,y)=\int e^{ikt\varphi(x,y)}\left(k^{n+1}\hat b_0(x,y',t,\tau)+k^n\hat b_1(x,y',t,\tau)+O(k^{n-1})\right)dt,
\end{split}
\end{equation}
where
    \begin{equation}\label{e-gue241206ycd}
    \begin{split}
     &\hat b_0(x,y',t,\chi)=\hat b_0(x,y,t,\chi)=b^\chi_0(x,y,t)=\sum^{+\infty}_{s=0}a_{0,s}(x,y')\chi^{(s)}(t)t^{n+s},\\
        &\hat b_1(x,y',t,\chi)=\hat b_1(x,y,t,\chi)=b^\chi_1(x,y,t)=\sum^{+\infty}_{s=0}a_{1,s}(x,y')\chi^{(s)}(t)t^{n+s-1},\\
     &a_{j,s}(x,y')\in\cC^\infty(D\times D),\ \ j,s=0,1,\ldots,\\
     &a_{1,s}(p,p)=0,\ \ s=1,2,\ldots, 
     \end{split}
    \end{equation}
    where $b^\chi_0(x,y,t)$, $b^\chi_1(x,y,t)$ are as in \eqref{Eq:LeadingTermMainThm}. 
    By the relation, 
\begin{equation}\label{eqn:relation tauchi=chi}
    \tau_k(T_P)\circ\chi_k(T_P)=(\tau\chi)_k(T_P)=\chi_k(T_P),
\end{equation}
and complex stationary phase formula, we have 
\begin{equation}
    \begin{split}
       &\tau_k(T_P)\circ\chi_k(T_P)(x,y)\\
       &=\int e^{ik(s\varphi(x,u)+t\varphi(u,y))}\Bigg[k^{2n+2}\hat b_0(x,u,s,\tau)\hat b_0(u,y,t,\chi)\\
       &\quad+k^{2n+1}\Big(\hat b_1(x,u,s,\tau)\hat b_0(u,y,t,\chi)+\hat b_0(x,u,s,\tau)\hat b_1(u,y,t,\chi)\Big)+O(k^{2n})\Bigg] V(u)dudsdt\\ 
       &=\int e^{ik(t(\sigma\varphi(x,u)+\varphi(u,y))}\Bigg[k^{2n+2}\hat b_0(x,u,t\sigma,\tau)\hat b_0(u,y,t,\chi)\\
&\quad+k^{2n+1}\Big(\hat b_1(x,u,t\sigma,\tau)\hat b_0(u,y,t,\chi)\\
&\quad+\hat b_0(x,u,t\sigma,\tau)\hat b_1(u,y,t,\chi)\Big)+O(k^{2n})\Bigg] tV(u)dud\sigma dt\\
       &=\int e^{ikt\varphi(x,y)}\Bigg[k^{n+1}\beta^{\chi,\tau}_0(x,y',t)+k^n\beta^{\chi,\tau}_1(x,y',t)+O(k^{n-1})\Bigg]dt,
    \end{split}
    \end{equation}
    where $\beta^{\chi,\tau}_0(x,y',t), \beta^{\chi,\tau}_1(x,y',t)\in\cC^\infty(D\times D\times I)$. From Theorem~\ref{t-gue241125yyda} (the uniqueness of the expansion) and \eqref{e-gue241206ycd}, we see that
    \begin{equation}\label{e-gue241206ycdr}
    \begin{split}
    &\beta^{\chi,\tau}_0(p,p,t)=\hat b_0(p,p,t,\chi)=b^\chi_0(p,p,t)=s_0(p,p)\chi(t)t^n,\\
    &\beta^{\chi,\tau}_1(p,p,t)=\hat b_1(p,p,t,\chi)=a_{1,1}(p,p)\chi(t)t^{n-1}.
    \end{split}
    \end{equation}
From \eqref{e-gue241206ycdr} and the stationary phase formula of H\"ormander, we have  
\[\begin{split}
    &k^na_{1,0}(p,p)t^{n-1}\chi(t)\\
    &= k^nt^{-n-1}L^{(1)}_{(u,\sigma)}\Big(\hat b_0(p,u,\sigma t,\tau)\hat b_0(u,p,t,\chi)e^{2(n+1)f(u)}\Big)|_{(u,\sigma)=(p,1)}\\
    &+(2\pi^{n+1})2k^na_{1,1}(p,p)t^{n-1}\chi(t)s_0(p,p)V(p)\chi(t),
\end{split}\]
where $L^{(1)}_{(u,\sigma)}$ is as in \eqref{e-gue241207ycdm}. 
From \eqref{eqn:ThmIII4} and notice that $s_0(p,p)=\frac{1}{2\pi^{n+1}}\left(\frac{V_{\xi}}{V}\right)(p)$, we can check that 
\begin{align}
    a_{1,0}(p,p)t^{n-1}\chi(t)=&-L^{(1)}_{(u,\sigma)}\Big(s_0(p,u)s_0(u,p)e^{2(n+1)f(u)}\tau(\sigma t)\sigma^n\Big)|_{(u,\sigma)=(p,1)}t^{n-1}\chi(t)\nonumber\\=&-
    L^{(1)}_{(u,\sigma)}\Big(s_0(p,u)s_0(u,p)e^{2(n+1)f(u)}\sigma^n\Big)|_{(u,\sigma)=(p,1)}t^{n-1}\chi(t).
\end{align}
Therefore, by \eqref{e-gue241121yyd}, we have
\begin{equation}
    a_{1,0}(p,p)=s_1(p,p).
\end{equation}
If we assume further $dV=dV_\xi$, then $f\equiv 0$, and thus
\begin{equation}
    a_{1,0}(p,p)=\frac{1}{4\pi^{n+1}}R_{{\rm scal\,}}(p).
\end{equation}
\end{proof} 
\begin{remark}\label{rmk:verifyargumentrmklocalinvariant}
	From the arguments in the proof of Theorem~\ref{t-gue241115yyd} one can see that for any \(j\geq 0\) we have that \(b_j(x,x,t)\) is given by a polynomial in the derivatives of \(b_0,\ldots,b_{j-1}\), \(\xi\) and \(dV\) at \((x,x,t)\). Hence, as a conclusion, we obtain the statement in Remark~\ref{rmk:ajLocalInvariants}.
\end{remark}
\begin{proof}[Proof of Theorem~\ref{t-gue241115ycd}]

 We now prove Theorem~\ref{t-gue241115ycd}. The first part of Theorem~\ref{t-gue241115ycd} follows immediately from Theorem~\ref{t-gue241123yyd} choosing \(P=-i\mathcal{T}\) assuming that \(dV\) is \(\mathcal{T}\)-invariant. Hence, we only need to prove \eqref{e-gue241116yyd}. 

So let us assume that $\mathcal{T}$ is a CR vector field. Let $x=(x_1,\ldots,x_{2n+1})$ be local coordinates defined on an open set $D$ of $X$ with $\mathcal{T}=\frac{\pr}{\pr x_{2n+1}}$ on $D$. From now on, we work on $D$ and we work with the local coordinates $x=(x_1,\ldots,x_{2n+1})$. Since $\mathcal{T}$ is a CR vector field, we can take $\varphi:D\times D\to\C$ satisfying 
\eqref{Eq:PhaseFuncMainThm}, \eqref{eq:PiFIO} and 
\begin{equation}\label{e-gue241210yyd}
\varphi(x,y)=x_{2n+1}-y_{2n+1}+\hat\varphi(x',y'),\ \ \hat\varphi(x',y')\in\cC^\infty(D\times D).
\end{equation}
We take $b^\chi_j(x,y,t)$ in \eqref{Eq:LeadingTermMainThm}, $j=0,1,\ldots$, so that \eqref{e-gue241128yydz} hold. 
Let $\reallywidetilde{\chi_k(A)}(x,y')$ be as in \eqref{eqn:ThmI3}. We have 
\begin{equation}\label{e-gue241210ycdy}
        \frac{1}{k}\reallywidetilde{\chi_k(A)(-i\mathcal{T})}(x,x')= \sum^{N}_{j=0}k^{n+1-j}m_0b^\chi_j(x,x',m_0)+O(k^{n-N}),
    \end{equation}
    for every $N\in\mathbb N$, for every $x\in D$. 
Let $\widehat{\chi}(t)=t\chi(t).$ We have $\frac{1}{k}\reallywidetilde{\chi_k(A)(-i\mathcal{T})}=\reallywidetilde{\widehat{\chi_k}(A)}$. From this observation and \eqref{e-gue241210ycdy}, we can repeat the proof of \eqref{eqn:ThmIII5} and deduce that 
\begin{equation}\label{e-gue241210ycdk}
a_{j,s}(x,x')=0, \ \mbox{for all $s=1,2,\ldots$, for all $j=0,1,\ldots$}, 
\end{equation}
where $a_{j,s}$, $j, s=0,1,\ldots$, are as in \eqref{e-gue241128yydz}. Thus, 
\begin{equation}\label{e-gue241210ycdl}
b^\chi_j(x,x',t)=a_{j,0}(x,x')\chi(t)t^{n-j},\ \ j=0,1,\ldots.
\end{equation}
Note that $a_{j,0}$ is independent of $\chi$, $j=0,1,\ldots$.
We have that $x\mapsto\chi_k(A)(x,x)$ is $\mathcal{T}$-invariant. From this observation and \eqref{e-gue241210ycdl}, it is not difficult to see that $b_j(x,x',t)$ is independent of $x_{2n+1}$, for all $j=0,1,\ldots$, and hence 
\begin{equation}\label{e-gue241210ycdm}
b^\chi_j(x,x',t)=b^\chi_j(x',x',t)=a_{j,0}(x',x')\chi(t)t^{n-j},\ \ j=0,1,\ldots.
\end{equation} 

From $\ddbar_{b,x}\varphi(x,y)$, $\ddbar_{b,x}\ol\varphi(y,x)$ vanish to infinite order at $x=y$, it is straightforward to check that  

\begin{equation}\label{e-gue241210ycdn}
\mbox{$\ddbar_{b,x}b_j(x,y',t)$, $\ddbar_{b,x}\ol b_j(y,x',t)$ vanish to infinite order at $x=y$}.
\end{equation}
From \eqref{e-gue241210ycdm} and \eqref{e-gue241210ycdn}, we can repeat the proof of~\cite{Hsiao_Galasso_JGA_2023}*{Lemma 3.1} with minor changes and deduce that there are $\alpha_j(x',y')\in\cC^\infty(D\times D)$, $j=0,1,\ldots$, such that 
\begin{equation}\label{e-gue241210ycdq}
\mbox{$b^\chi_j(x,y',t)-\alpha_j(x',y')\chi(t)t^{n-j}$ vanishes to infinite order at $x=y$, $j=0,1,\ldots$}.
\end{equation}
From \eqref{e-gue241210ycdq}, we can change $b^\chi_j(x,y',t)$ to $\alpha_j(x',y')\chi(t)t^{n-j}$, for all $j=0,1,\ldots$. The theorem follows. 
\end{proof}

\bibliographystyle{plain}

\begin{thebibliography}{99} 
\bibitem{Boutet_Sjoestrand_Szego_Bergman_kernel_1975}
L.~Boutet de Monvel and J.~Sj{\"o}strand. \emph{{Sur la singularit\'{e} des noyaux de Bergman et de Szeg\H{o}}.} {Journ{\'e}es {\'E}quations aux D{\'e}riv{\'e}es Partielles} (1975) 123-164. Ast\'{e}risque, No. 34--35. 

\bibitem{Catlin_BKE_1999}
D.~Catlin, \emph{The Bergman kernel and a theorem of Tian},
Analysis and Geometry in Several Complex Variables, Birkh\"auser, Boston, (1999), pages 1--23. 

\bibitem{Donaldson_cscK_2001}
  S.~Donaldson, \emph{Scalar curvature and projective embeddings I},
Journal of Differential Geometry, volume 59, no, 3, (2001), pages 479--522. 

\bibitem{Hsiao_Galasso_JGA_2023}
A. Galasso and C.-Y. Hsiao, \emph{Toeplitz operators on CR manifolds and group actions}, J. Geom. Anal. {\bf 33} (2023), no.~1, Paper No. 21, 55 pp.; MR4510165 

\bibitem{Herrmann_Hsiao_Li_Q-R-Sasakian_Szego_2018}
H.~Herrmann, C.-Y.~Hsiao and X.~Li, \emph{{Szeg\H{o} kernel asymptotic expansion on strongly pseudoconvex CR manifolds with \(S^1\) action.
}} Int. J. Math. (2018), 1850061.

\bibitem{Herrmann_Hsiao_Li_T_equivariant_Szego_2020}
H.~Herrmann, C.-Y.~Hsiao and X.~Li, \emph{{Torus equivariant Szeg\H{o} kernel asymptotics on strongly pseudoconvex CR manifolds.
}} Acta Math. Vietnam. 45 (2020), no. 1, 113–135.

\bibitem{HHMS23}
H.~Herrmann, C.-Y.~Hsiao, G.~Marinescu and W.-C.~Shen.
\emph{Semi-classical spectral asymptotics of Toeplitz operators on CR manifolds},  arXiv preprint arXiv:2303.17319 (2023). 

\bibitem{Hsiao_Szego_Bergman_kernel_q_forms_2010}
C.-Y.~Hsiao,
\emph{{Projections in several complex variables.}} Mémoires de la Société Mathématique de France, 123 (2010), 131 pages. 

\bibitem{Hsiao_BT_coefficient_2012}
C.-Y. Hsiao, \emph{On the coefficients of the asymptotic expansion of the kernel of Berezin-Toeplitz quantization}, Ann. Global Anal. Geom. {\bf 42} (2012), no.~2, 207--245; MR2947953 

\bibitem{HM14} 
C.-Y.~Hsiao and G.~Marinescu.
\emph{Asymptotics of spectral function of lower energy forms 
and Bergman kernel of semi-positive and big line bundles},
Communications in Analysis and Geometry, {\bf 22} (2014), no. 1, 1-108. 

\bibitem{Hsiao_Marinescu_Szego_lower_energy_2017}
	C.-Y.~Hsiao and G.~Marinescu.
	\emph{On the singularities of the Szeg\H{o} projections on lower energy forms},
	J. Differential Geometry 107 (2017) 83-155. 

    \bibitem{Hsiao_Shen_2nd_coefficient_BS_2020}
C.-Y.~Hsiao and W.-C.~Shen. \emph{On the second coefficient of the asymptotic expansion of Boutet de Monvel-Sj\"ostrand.} Bulletin of the Institute of Mathematics Academia Sinica NEW SERIES 15(4), 2020.

\bibitem{Ma_Marinescu_HMI_BKE_2007}
	X.~Ma and G.~Marinescu.
	\emph{Holomorphic {M}orse inequalities and {B}ergman kernels}, volume
	254 of Progress in Mathematics, Birkh\"auser Verlag, Basel, 2007.
	 
\bibitem{Ma_Marinescu_JRAM_2012}
X. Ma and G. Marinescu, \emph{Berezin-Toeplitz quantization on  
	K\"ahler manifolds}, J. Reine Angew. Math. {\bf 662} (2012), 1--56; MR2876259
	
\bibitem{Zelditch_BKE_1998}
S.~Zelditch, \emph{Szeg{\H o} kernels and a theorem of {T}ian},
International Mathematics Research Notices, 1998, no. 6, 317--331. 	

\end{thebibliography}

\end{document}